\documentclass[12pt]{amsart}

 \usepackage{graphicx}
 \usepackage{amscd, color}

\usepackage{marginnote}

\swapnumbers

 \def\al{\alpha}
 \def\be{\beta}
 \def\de{\delta}
 
 \def\eps{\varepsilon}

 \def\ga{\gamma}
  \def\G{\Gamma}
 \def\GG{\mathcal G}

 \def\la{\lambda}
 
 \def\si{\sigma}
 \def\om{\omega}
 
 \def\d{\mathrm d }
 \def\EE{{\mathbf E}}
 \def\VV{{\mathbf V}}
  
 \def\EN{{\mathcal E}}

 \def\ie{{\it i.e.}, }
  \def\G{\Gamma}

 \def\R{{\mathbb R}}
 \def\Z{{\mathbb Z}}

  \def\A{{\mathcal  A}}
 
 \def\LL{{\mathcal  L}}
 \def\HH{{\mathcal  H}}
 
 \def\MM{{\mathcal M}}
 \def\ov{\overline}

\def\Cf {\mathfrak C}

 \def\oo{\mathrm o}
 \def\tt{{\mathrm t}}

 \def\bM {\mathbb M}

\def \txt{\qquad\hbox}

 \DeclareMathOperator{\<}{\langle}
 \DeclareMathOperator{\ra}{\rangle}

 \DeclareMathOperator{\supp}{supp}

  \renewcommand{\proofname}{{\bf Proof:}}

 \theoremstyle{plain}

  \swapnumbers

 \newtheorem{Thm}{Theorem}[section]
 
 \newtheorem{Lemma}[Thm]{\bf Lemma}
 
 \newtheorem{Corollary}[Thm]{\bf Corollary}
 \newtheorem{Theorem}[Thm]{\bf Theorem}
 \newtheorem{Proposition}[Thm]{\bf Proposition}

 \theoremstyle{definition}
 \newtheorem{Definition}[Thm]{\bf Definition}

 \theoremstyle{remark}
 \newtheorem{Remark}[Thm]{\bf Remark}

 \newtheoremstyle{Cl}% name
  {5pt}%      Space above
  {3pt}%      Space below
  {\sl}%   Body font
  {}%         Indent amount (empty = no indent, \parindent = para indent)
  {\it}% Thm head font
  {:}%        Punctuation after thm head
  {.5em}%     Space after thm head: " " = normal interword space;
  {}%         Thm head spec (can be left empty, meaning `normal')

 \theoremstyle{Cl}

 \def\begincproof{
                  \renewcommand{\proofname}{\it Proof:}
                  \begin{proof}
                 }

 \def\endcproof{
                \renewcommand{\qedsymbol}{$\diamondsuit$}
                \end{proof}
                \renewcommand{\qedsymbol}{\openbox}
                \renewcommand{\proofname}{\bf Proof:}
               }

 \renewcommand{\proofname}{{\bf Proof:}}

\title[Aubry--Mather]
{Aubry--Mather theory on graphs}

\author{Antonio Siconolfi}
\address{Dipartimento di Matematica, Sapienza Universit\`a  di Roma, Italy.}
\email{siconolfi@mat.uniroma1.it}

\author{Alfonso Sorrentino}
\address{Dipartimento di Matematica, Universit\`a degli Studi di Roma ``Tor Vergata'', Rome, Italy.}
\email{sorrentino@mat.uniroma2.it}

\subjclass[2010]{}
\keywords{}

\begin{document}
\maketitle

 \begin{abstract} We formulate  Aubry--Mather theory for Hamiltonians/Lagrangians defined on graphs and discuss its relationship with weak KAM theory developed in \cite{SiconolfiSorrentino}.
\end{abstract}

\section{Introduction}

The Hamilton Jacobi equation and its wide spectrum of applications represent an important crossroads for ideas and techniques coming from different areas of mathematics: PDEs, calculus of variations, control theory, optimal transport,  symplectic geometry, etc...
Over recent years, this beneficial synergy has resulted in the development of novel lines of research and significant scientific advances.\\

In 1990's this multifaceted interaction has experience a particular boost thanks to a remarkable intuition by Albert Fathi and the ensuing development of what is nowadays called {\it weak KAM theory} \cite{Fathi, FathiICM}. This novel point of view shed light on the noteworthy connection between the very successful action-minimizing methods in Hamiltonian dynamics
-- in particular {\it Aubry-Mather theory}, originated from the work by Serge Aubry \cite{Aubry}  and John Mather \cite{Math82, Mather91} --, and the analysis of viscosity solutions and sub-solutions to the Hamilton-Jacobi equation.

Not only these ideas were particularly beneficial for enhancing our understanding of both the dynamics of these systems and the global properties of the solutions, but they also contributed to draw unexpected connections with other problems in dynamics, geometry and analysis (see for instance \cite{FathiICM, Siburg, Sorrentinobook}). \\

These successes  encouraged a very active investigation on the possibility of extending these theories beyond the classical settings, either to other classes of systems or to  different ambient spaces, more suitable for some applications.

 In this article, it is pursued the latter direction.\\

The paper presents the first, as far as we know, systematic detailed account of Aubry--Mather theory for Hamiltonians/Lagrangians defined on {\it graphs} recovering the whole theory in this new context and relating it to  weak KAM analysis carried out in \cite{SiconolfiSorrentino}.

\subsection*{Motivations and Significance}

Over the last years there has been an increasing interest in the
study of the Hamilton-Jacobi equation on {\it graphs} and {\it networks} and related
questions. These problems, in fact, besides having a great impact in the applications in
various fields (for example to data transmission, traffic management
problems, etc...), they  involve a number of subtle
theoretical issues related to the intertwining between the local analysis of the problem and the global structure of the network/graph.\\

Several reasons can be advanced for  embarking  on the job of settling Aubry-Mather theory in this setting. \\

On the one hand, the need of following up the line of investigation initiated in \cite{SiconolfiSorrentino}, where we provided a thorough discussion of KAM theory on graphs/networks (see also \cite{PoSi}).
Weak KAM  and Aubry-Mather theories, in fact,  are in a duality relationship: they are both intrinsically based on the study of objects that arise as  minimizers of some {\it action functionals}, possibly with constraints (see for instance \cite{FathiICM, Siburg, Sorrentinobook} for more details).\\

 On the other hand, the passage from manifolds to graphs requires a specific adaptation of the main tools and techniques involved ({\it e.g.}, parametrized paths, spaces of  probability measures, occupation and closed probability measures, etc...) which is by no means straightforward  and, we believe, would be of potential interest for many other problems and applications.\\

As a matter of fact, one of the initial motivation for this work,  was to prove a  homogenization result for  the Hamilton--Jacobi equations on  networks, following  the homological approach introduced in \cite{CIS}.  However, in order to pursue the project, it became crucial to first develop an Aubry-Mather theory in this context, both for determining the limit problem and the space on which it is defined.
For instance,  the  effective Hamiltonian    appearing, as in the compact manifold case, in the limit equation is nothing else that  the {\it minimal average action} (or  {\it Mather's $\alpha$ function}, see subsection \ref{sec5.2}), namely the value function related to one of the variational problem at the core of Aubry Mather theory.

 Note that, even if the approximated equations in the homogenization problem are  posed on a network, the natural setting where the approximation procedure should take place is the corresponding abstract graph. This is one of the reasons why the present contribution is focused on this issue.\\

To avoid misunderstanding, we make it clear that the topic we are talking about -- that has not  been treated in the literature yet --, is different from  other interesting models  of partial homogenization on junctures considered in  \cite{FS, GIM},  and mainly devoted to applications to traffic theory. \\

We remark that a central role in our construction   is played by  probability measures, defined on a sort of tangent bundle of the graph: they constitute the relaxed framework for the variational problems under consideration. There is  a broad interest in the  recent literature on probability measures  supported on graphs/networks, see for instance \cite{CAMI2, Maz}. One of the goal being, for instance,  to extend   mean field games models to graphs (see  \cite{Achdou, CAMI, Gangbo, GU}).  Passing to a related field, connections between Aubry--Mather theory  and optimal transport have been pointed out by various author, see \cite{BB1, BB2, Bernard}. \\
The outputs of the present paper can be seen as  a first step to explore these  directions of research in the graph setting.

\subsection*{Main contributions}
In this paper we prove that   Aubry--Mather theory can be completed established in the graphs context. In particular, we discuss: the role of  occupation and closed probability measures, existence of action minimizing measures, properties of the corresponding  value functions (Mather's $\alpha$ and $\beta$ functions), properties of Mather measures, structure of Mather sets, Mather graph property, etc....
\\

We recover the duality links with weak KAM theory as well.  Weak KAM theory, as formulated in \cite{SiconolfiSorrentino}, is essentially a metric theory  based on the notion of   intrinsic length of paths, which in turn depends on a given level of the Hamiltonian; we do not even need a Lagrangian function to be defined. The sign of the intrinsic length of cycles is related to the existence of subsolutions   to the corresponding Hamilton-Jacobi equation. The existence of cycles with vanishing intrinsic length is attained  at the critical value of the Hamiltonian, which is the unique value for which one finds solutions to the Hamilton--Jacobi equations. The edges forming cycles with vanishing length make up the so--called {\it Aubry set}.\\

 Aubry-Mather theory is  instead a  variational theory, inspired by the  principle of least action \cite{Mather91, Siburg, Sorrentinobook},  whose aim is to find, in suitable spaces, minimizers of the {\it Lagrangian action functional} (possible with constraints). \
 Our  starting point is to introduce the notion of parametrization of a  path (subsection \ref{paths}),  which  is in duality with that of intrinsic length. It  is obtained  by  equipping every edge of a path  with a non-negative weight, that can be interpreted as an average speed. This allows us to introduce  an action functional on the set of paths,  setting up the variational problem of interest. \\

  The usual relaxation procedure yields to pose the problem in a suitable  space of measures, where all the minimizers can be found. To this aim we define, on an appropriate tangent bundle of the graph, the notion of closed occupation measures, which somehow correspond to parametrized cycles, and prove that they are dense, with respect to the first Wasserstein metric, in the space of all closed probability measures (Appendix \ref{B}), which  we believe has its  own interest beyond the problem at hand. Significantly, the measures minimizing the action are supported by cycles with vanishing intrinsic length and the minimum of the action is equal, up to a sign, to the critical value (see Theorem \ref{kamalaa}).

\subsection{Organization of the article}
We describe hereafter  how the article  is organized. \\

In {\bf Section \ref{networks}}  we provide a brief introduction to graph theory,  in order to set the terminology and introduce the main concepts  that will be needed. In particular,  we define the algebraic topological notions of chains, cochains, homology and cohomology of the graph, that are crucial importance for the full implementation of the variational analysis.  \\

In {\bf Section \ref{sec3}} we give the notion of Hamiltonian on a graph and introduce the associated Lagrangian  which allows us to define the action functional to be minimized under appropriate constraints.\\

In {\bf Section \ref{misura}}  we define the relaxed setting on which the variational analysis will occur.
Then, we introduce the notion of {\it occupation measures} and {\it closed probability measures}, and the relevant Wasserstein topology. These are central objects in Aubry-Mather theory
that represent  useful relaxations of  the notion of paths and closed paths.\\

{\bf Sections \ref{secMather} \& \ref{sec6}} are the core of the development of Aubry-Mather theory in the context of  graphs. We set, in analogy to the classical setting, a family of variational problems, show that they admit global minimizers and discuss their significance and their structural properties.  Interestingly,  we prove in this context the analogue of the celebrated  Mather's graph theorem (Proposition \ref{mathergraph2} and Corollary \ref{ironic}). \\

After having recalled in {\bf Section \ref{sec7}} the basic results of weak KAM theory from \cite{SiconolfiSorrentino}, in {\bf Section \ref{WKAM}}
we discuss the relation between  Aubry Mather theory and weak KAM theory on graphs.  We show the equality between the (projected) Mather sets and the corresponding Aubry sets, and in Theorem \ref{proHJomega} we use viscosity solutions and subsolutions to provide a more explicit description of Mather's graph theorem.\\

%After having recalled in {\bf Section \ref{sec7}} the basic results of weak KAM theory from \cite{SiconolfiSorrentino}, in {\bf Section \ref{WKAM}} we
%we discuss the relation between  Aubry Mather theory and weak KAM theory on graphs. As in the classical case, these two approaches turn out to be tightly intertwined, each providing a different and interesting perspective on the other. More specifically:
%\begin{itemize}
%\item[-] In Theorem \ref{kamalaa} we prove that Mather's $\alpha$ function is related to the critical value of certain Hamilton-Jacobi equation, and  that (irreducible) Mather measures are supported on circuits of vanishing intrinsic length.
%\item[-] In Corollary \ref{cor8.3} we show the equality between the (projected) Mather sets and the corresponding Aubry sets and in Theorem \ref{proHJomega} we use viscosity solutions and subsolutions to provide a more explicit description of Mather's graph theorem.
%\item[-] Finally, in subsection \ref{singular} we discuss some properties of Mather measures corresponding to the minimum of Mather's $\alpha$-function.\\
%\end{itemize}

Finally, in {\bf Appendix \ref{netto}} we describe how to develop  an Aubry--Mather theory on networks, and  look from the point of view of networks to some notions we have introduced  on graphs. In {\bf  Appendix \ref{B}} we provide the proof of the density result of closed occupation measures.

\bigskip

\subsection*{Acknowledgments}
The second author
acknowledges the support of
the University of Rome Tor Vergata's {\it Beyond Borders} grant  ``{\it The Hamilton-Jacobi Equation: at the crossroads of Analysis, Dynamics and Geometry}'' (CUP: E84I19002220005) and the
Italian Ministry of Education and Research (MIUR)'s grants:
PRIN Project  ``{\it Regular and stochastic behavior in
dynamical systems}'' (CUP: 2017S35EHN)  and  the Department of Excellence grant  2018-2022 awarded to the Department of Mathematics of University of Rome Tor Vergata   (CUP: E83C18000100006).

Finally, both authors with to express their gratitude to the  Mathematical Sciences Research Institute in Berkeley (USA) for its kind hospitality in Fall 2018 during the trimester program ``{\it Hamiltonian systems, from topology to applications through analysis}'',
where part of this project was carried out.

\bigskip

\section{Prerequisites on graphs}\label{networks}

\subsection{Definition and terminology} A  graph $\Gamma=(\VV ,\EE )$ is an
ordered pair of disjoint  sets $\VV$ and $\EE$,  which are  called,
respectively, {\it vertices} and (directed) {\it edges}, plus two
functions:
$$\oo: \EE \longrightarrow \VV $$
which associates to each  edge its {\it origin} (initial
vertex), and
\begin{eqnarray*}
 - {\phantom{o}}: \EE &\longrightarrow& \EE \\
e &\longmapsto& - e,
\end{eqnarray*}
which  changes  direction and  is a
fixed point free involution, namely
\[- e \neq e \txt{and} \qquad -(-e)=e \txt{for any $e \in \EE$.}\]
We  define the {\em terminal vertex} of $e$ as
\[\tt  (e):= \oo ( - e).\]
We further denote by $|\VV|$, $|\EE|$, the number   of vertices and edges, respectively.  For any vertex $x\in\VV$, we denote by
$$
\EE_x :=\{ e\in \EE:\; \oo(e)=x \}
$$
the set of edges originating from $x$; this is sometimes called the {\it star centered at $x$}.\\

An {\it orientation} of $\Gamma$ is a subset $ \EE^+$ of the edges satisfying
$$ -\EE^+ \cap \EE^+ = \emptyset \qquad {\rm and}\qquad
-\EE^+  \cup \EE^+ = \EE.$$
In other words, an orientation of $\Gamma$ consists of a choice of exactly one edge in each pair $\{e,-e\}$.\\

We define a {\it path}  $\xi:=(e_1, \cdots, e_M) = (e_i)_{i=1}^M$ as a finite sequence
of concatenated  edges in $\EE$, namely
$\tt  (e_j)=\oo (e_{j+1})$ for any $j= 1, \cdots, M-1$.\\
We  define the {\em length of a path} as the number of its  edges. We set $\oo (\xi):= \oo (e_1)$, $\tt  (\xi):= \tt  (e_M)$.
We call a path {\it closed}, or a {\it cycle}, if $\oo (\xi)= \tt
(\xi)$. \\

Throughout the paper, we assume $\G$ to be

\begin{itemize}
  \item[{\bf (G1)}] { \em finite}, namely with $|\EE|$, $|\VV|$ finite;
  \item[{\bf (G2)}] {\em connected}, in the sense that any two vertices are linked by some path;
  \item [{\bf (G3)}] {\em without loops}, namely for any $e \in \EE$ $\oo(e) \neq \tt(e)$.
\end{itemize}

The first two assumptions are structural, while the last one  could be removed at the price of introducing further details in the development of the theory. For the sake of clarity of this presentation, we prefer to avoid it in the present paper.\\

It follows from the connectedness assumption, that the functions $\oo$ and $\tt$ are surjective. \\

We call  {\it simple} a path without repetition of vertices,  except
possibly the initial and terminal vertex, in other terms
$\xi=(e_i)_{i=1}^M$ is simple if
\[\tt(e_i) = \tt(e_j) \, \Rightarrow i=j.\]
Clearly, there are finitely many simple paths  in  a finite graph.
We call {\em circuit} a simple closed path.
{Given any edge $e$, we call {\em equilibrium circuit} (based on $e$) the path $(e,-e)$}.

\medskip

\subsection{Homology of a graph}

Throughout the paper we will  take homology and cohomology with coefficients in $\R$. We refer to \cite[Ch. 4]{Sunada} for a more detailed and general presentation.\\
We define the {\it$0$--chain} group as the free Abelian group on the vertices with coefficients in  $\R$.
We denote it by
$\Cf_0( \G, \R)$.   We have
\[\ \Cf_0( \Gamma,\R) \sim \R^{|\VV|}.\]

We do the same operation with edges,  making   the  reversed edge $-e$ coincide with the opposite of $e$ with respect to the group operation, and we obtain the {\it $1$--chain group}, denoted  by  $\Cf_1(
\G, \R)$.  A basis is given by any orientation  $\EE^+$, in the sense the any element of the $1$--chain group can be uniquely expressed as a linear combination of  elements in $\EE^+$ with  real  coefficients.  We consequently have
\[ \Cf_1( \G,\R) \sim \R^{|\EE|/2}.\]
We define the {\it boundary operator} $\partial: \Cf_1( \G,\R) \to \Cf_0( \G,\R)$  by setting for any edge
\[\partial e:= \tt  (e)- \oo (e)\]
and then extending it linearly; clearly, $\partial\, (-e)= -\partial e$.\\

The ({\it first}) {\it Homology group} of $\G$ with coefficients
in $\R$ is defined by
$$H_1(\G,\R):= {\mathrm Ker}\,\partial,$$
Some remarks:
\begin{itemize}
  \item[--] $H_1(\G,\R)$ is a subgroup of $\Cf(\G,\R)$.
  \item[--] $H_1(\G, \R)$ is a free Abelian group of finite rank. The ({\it first}) {\it Betti number}  is defined to be the rank of $H_1(\G, \R)$, it is an indicator of the topological complexity of the network.
  \item[--] An element of $H_1(\G,\R)$ is called a {\it $1$--cycle}. In particular a $1$--chain $\sum_{e \in \EE^+} a_e e$ is a $1$--cycle if and only if
      \begin{equation}\label{cincin}
       \sum_{e \in \EE^+,\, \tt(e)=x} a_e = \sum_{e \in \EE^+,\,\oo(e)=x} a_e \txt {for any $x \in \VV$;}
      \end{equation}
This can be considered as an analogue of {\it Kirchhoff law} for electric circuits.
\end{itemize}
  Due to \eqref{cincin}, we can associate to any closed path $\xi=(e_i)_{i=1}^M$  in $\Gamma$ an element of $H_1(\G,\R)$ via
  \begin{equation}\label{cincinbin}
   [\xi] := \sum_{i=1}^M e_i.
  \end{equation}
  We call $[\xi]$ the {\it homology class} of $\xi$. The converse is also true: every element of $H^1(\Gamma,\Z)$  can be represented  by a closed path  (see \cite[pp. 40--41]{Sunada}).
\medskip

\subsection{Cohomology of a graph}\label{subseccohom}

 Let us introduce the dual entities of chains. The {\it $0$--cochain
group}, denoted by  $\Cf^0( \Gamma, \R)$,  is  the
space of functions from $\VV$ to  $\R$,  and  the
{\it $1$--cochain group}, denoted by  $\Cf^1( \Gamma, \R)$, is the
space of functions $\eta: \EE \longrightarrow \R$, satisfying
the compatibility condition
\[\eta(- e)= - \eta(e) \txt{for any  $e \in \EE$.}\]
The algebraic structure of additive Abelian group is induced by the one in  $(\R,+)$.\\

We introduce the  {\it differential} or {\it coboundary operator}
$$
d: \Cf^0(\G, \R) \longrightarrow \Cf^1(\G, \R)
$$
 which is defined in the following way: for every $g \in \Cf^0(\G,\R)$, the $1$--cochain $dg$ is given via
$$dg(e):= g (\tt(e)) - g (\oo(e)) \txt{ for all  $e\in \EE$;}$$
it clearly satisfies the compatibility condition $dg(- e) = -dg (e)$.
\\

It is easy to check that $d$ is a group homomorphism. Hence,
the ({\it first}) {\it Cohomology group} of $\G$ with
coefficients in $\R$  can be defined as the quotient group
$$H^1(\G,\R):= \Cf^1(\G,\R) / {\mathrm {Im}}\,d.$$
One can show that there exists a canonical isomorphism
$$
H^1(\G,\R) \simeq {\rm Hom}\left( H_1(\G,\R),\R \right).\\
$$

\bigskip

\subsection{Pairings between chains and cochains, homology and cohomology} \label{subsecpairing}

Let us   introduce a
{\it pairing} between $0$--chains and $0$--cochains:
\begin{eqnarray*}
\langle \cdot, \cdot \rangle: \Cf^0(\G,\R) \times \Cf_0(\G,\R) &\longrightarrow& \R\\
\left (g, \sum_{x\in \VV} \al_x x \right ) &\longmapsto&   \sum_{x\in \VV} \al_x g(x).\\
\end{eqnarray*}

Similarly,
we can define the pairing  between $1$--chains and $1$--cochains (we adopt the same notation):
\begin{eqnarray*}
\langle \cdot, \cdot \rangle: \Cf^1(\G,\R) \times \Cf_1(\G,\R) &\longrightarrow& \R\\
\left (\eta, \sum_{e\in \EE} \al_e e \right ) &\longmapsto&   \sum_{e\in \EE} \al_e \eta(e).\\
\end{eqnarray*}

The above pairings allow us to relate differential and boundary operators. Let $g\in \Cf^0(\G,\R)$ and $\zeta= \sum_{e\in \EE} \al_e e \in \Cf_1(\G,\R)$; then we have:
 \begin{eqnarray}\label{pairing}
\langle dg ,  \zeta \rangle &=&
\sum_{e\in \EE} \al_e dg (e) =
\sum_{e\in \EE} \al_e \big(  g(\tt(e)) - g(\oo(e)) \big) \nonumber\\
&=& \sum_{e\in \EE} \al_e \langle  g , \partial e  \rangle  =
\langle g, \sum_{e\in \EE} \al_e e \rangle =
\langle g ,  \partial \zeta \rangle.
\end{eqnarray}

In particular, this means that whenever $\zeta\in \Cf_1(\G,\R)$ is such that $\partial \zeta =0$, then $\langle dg ,  \zeta \rangle =0$ for all $g\in \Cf^0(\G,\R)$.  Hence, the above pairing descends to a well-defined pairing between first homology and first cohomology groups, that we continue to denote
$\langle \cdot, \cdot \rangle: H^1(\G,\R) \times H_1(\G,\R) \longrightarrow \R.$

\bigskip
\section{Hamiltonians and Lagrangians on graphs} \label{sec3}

\subsection{Definitions and assumptions}  We call a {\it Hamiltonian} on the  graph $\Gamma= (\VV, \EE)$   a family of  functions
\[ \mathcal H(e,\cdot): \R  \to \R\]
labeled by the edges, such that
\begin{equation}\label{lag1}
 \HH(e, p)= \HH(-e,-p) \txt{for any $e \in \EE$, $p \in \R$.}
\end{equation}
We further require that, for any $e \in \EE$, $\HH(e,\cdot)$ is
\begin{itemize}
    \item[{\bf (H1)}]  {\em strictly convex} and {\em differentiable};
    \item[{\bf (H2)}] {\em superlinear}  at $\pm \infty$, namely
    \[ \lim_{p \to \pm \infty} \frac{\HH(e,p)}{|p|}= + \infty.\]
\end{itemize}

This implies that there exists, for any $e$, a unique $p_e=-p_{-e}$ global minimizer of both $\HH(e,\cdot)$  in $\R$. We consider in what follows $\HH(e,\cdot)$ mostly restricted to $[p_e,+ \infty)$, (resp. $\HH(-e,\cdot)$ restricted to $[p_{-e},+ \infty)$),  which is strictly increasing in  this  domain of definition.  We set
\begin{equation}
  a_e =\HH(e,p_e) = \HH(-e,p_{-e})=a_{-e} \label{ae}
\end{equation}
We define  $\si(e,\cdot)$ as the inverse function of $\HH(e,\cdot)$ in $[p_e,+ \infty)$. We have
\[\si(e,\cdot): [a_e, +\infty) \to [p_e,+\infty) \txt{for any $e \in \EE$}\]
and
\begin{equation}\label{cumpa}
 \si(e,a_e)= - \si(-e, a_e)=p_e =- p_{-e}   \txt{for any $e$}.
\end{equation}

The properties summarized in the next statement are immediate.

\begin{Lemma}\label{newbornbis} Let $e \in \EE$.  The function $a \mapsto \si(e,a)$  from
$[a_e,+\infty)$  to $\R$ is continuous, differentiable in $(a_e,+\infty)$,  and strictly increasing for any $e$. In addition, it is strictly concave and satisfies
\[ \lim_{a \to + \infty} \frac{\si(e,a)}a =0.\]
\end{Lemma}

\smallskip

We define the {\it Lagrangian} $\LL(e,\cdot): \R \to \R$ as the convex conjugate of $\HH(e,\cdot)$, namely
\[\LL(e,q):= \max_{p \in \R} \big (p \, q - \HH(e,p) \big ).\]

\begin{Proposition} \label{swanlake}Let $e \in \EE$. The function $q \mapsto \LL(e,q) $ is strictly convex and superlinear as $q$ goes to $\pm \infty$. In addition
\begin{equation}\label{lag2}
  \LL(e,q)= \LL(-e,-q) \txt{for any $q \in \R$.}
\end{equation}
\end{Proposition}

\medskip

\noindent This is a consequence of  {\bf (H1)--(H2)}  and \eqref{lag1}  (see, for instance, \cite[Theorem 26.6]{Rockafellar}).\\

\smallskip

In what follows, we mostly consider $\LL(e,\cdot)$ restricted to $[0,+\infty)$. We have
\[\LL(e,q)=  \max_{p \geq \si(e,a_e)} \big (p \, q - \HH(e,p) \big ) \txt{for $q \geq 0$,}\]
an equivalent formula is
\begin{equation}\label{traviata}
\LL(e,q)=  \max_{a \geq a_e}  \big ( q \,\si(e,a) - a \big ) \txt{for $q \geq 0$,}
\end{equation}
from which it follows that $\LL(e,0)=-a_e$.\\

Given   $\om \in \Cf^1(\Gamma,\R)$,  we further consider the { \it $\omega$-- modified Hamiltonian}
\[\HH^\om(e,p):= \HH(e,p + \langle \om, e \rangle),\]
which  clearly still satisfies assumptions {\bf (H1)}, {\bf (H2)}. It is therefore  invertible on the right of its minimizer and the inverse is
\begin{equation}\label{sigmaom}
 \si^\om(e,a):= \si(e,a)- \langle \om, e \rangle.
\end{equation}
The corresponding { \it $\omega$--modified Lagrangian} is given by
\[\LL^\om(e,q) := \LL(e,q) -  \langle \om, q e \rangle.\]

\smallskip

\begin{Remark}\label{un}
 Note that  $a_e$ does not depend on $\omega$, {\it i.e.}, it is the same for  $\HH^\omega(e,\cdot)$. In fact by \eqref{cumpa}  $a_e$ is characterized by the relation
\[\si(e,a_e) + \si(-e,a_e) =0\]
and by \eqref{sigmaom}
\[\si(e,a_e) + \si(-e,a_e)=\si^\om(e,a_e) + \si^\om(-e,a_e) \txt{for any $1$--cochain $\om$.}\]
\end{Remark}

\medskip

\bigskip

 \section{Probability measures on edges}\label{misura}

\subsection{Preamble: parametrized paths} \label{paths}
  The notion of {parametrized path} is central in the paper and it will be essential  to define occupation measures. \\

   Intuitively speaking, a parametrized path is a path where it is assigned to any edge a non-negative  {\it average speed}  and  a  {\it  time} needed to go through it. The time is the inverse of the speed, if the latter  is positive, while it can be any possible positive number if the speed is zero. We motivate this choice in Section \ref{dragone} in the case where $\G$ is the abstract graph associated to a network.

\begin{Definition} \label{para}
  We say that
 $\xi=(e_i,q_i,T_i)_{i=1}^M$  is  a {\it parametrized path}  if
 \begin{itemize}
   \item[{\bf (i)}] $(e_i)_{i=1}^M$ is a family of concatenated edges which is called the {\em support} of $\xi$;
   \item[{\bf (ii)}] the  $q_i$ are non-negative
 numbers and
\[ T_i =    \left \{ \begin{array}{cc}
                     \frac 1{q_i} & \quad\hbox{if $q_i > 0$} \\
                       \hbox{a positive constant} &  \quad\hbox{if $q_i = 0$};
                   \end{array}  \right . \]
we denote by $ T_\xi:=\sum_i T_i$ the {\em total time} of the parametrization of $\xi$;
   \item[{\bf (iii)}]  if all the   $q_i's$ vanish then $\oo(\xi)= \tt(\xi)$;
   \item[{\bf (iv)}]  if $q_i=0$ and $e_{i+1} \neq -e_i$ then $q_{i+1} \neq 0$;
   \item[{\bf (v)}] if $q_i \neq 0$, $i >1$, then
   \[\oo(e_i)=  \tt(e_j)  \txt{ with $ j =\max \{k < i, \, q_k \neq 0\}$.}\]
   \end{itemize}
    We call a {\em parametrized cycle}, a parametrized path supported on a closed path (or cycle).
 We call a {\em parametrized circuit}, a parametrized path supported on a circuit.\\
 \end{Definition}

\begin{Remark}
Intuitively, a parametrized path can be thought as a concatenation of  triples with non-zero average velocity, and  pairs of triples ({\it i.e.},  {\it equilibrium circuits}) of the form $ \{(e,0,T), (-e,0,S)\}$ for some $e\in \EE$ and $T,S>0$. In particular, condition {\bf (iv)} reads that  there cannot be consecutive equilibrium circuits corresponding to different edges.\\
Equilibrium circuits represent steady states, interpreted as floating with zero average speed  along  an edge and its opposite.
Therefore, if all  speeds vanish (item {\bf (iii)}) then initial and final position must coincide.
Items  {\bf (iv)}, {\bf (v)} further  prescribe   that an object possessing  vanishing speed on an edge $e$ starts floating back and forth along $e$ and $-e$, and exits the swinging state from the same vertex it entered,  only  when the speed becomes positive.
\end{Remark}

We deduce from the definition the following properties.

\begin{Proposition}\label{capizzi}
Let $\xi=(e_i,q_i,T_i)_{i=1}^M$ be a parametrized path.
\begin{itemize}
\item[{\bf (i)}] If some speed  $q_i$ is non-vanishing, and  $i_1, \cdots , i_K$  is the increasing sequence of indices corresponding to edges with positive speed,  then
       \[\bar \xi:= (e_{i_j}, q_{i_j}, T_{i_j})_{j=1}^K\]
  is  still a parametrized path with all average velocities different from $0$ and such that $\oo(\bar\xi)= \oo(\xi)$, $\tt(\bar\xi)= \tt(\xi)$.
  \item[{\bf (ii)}] If a  parametrized path has all  average speeds equal to zero, then it is supported on an edge and its opposite.
  \item[{\bf (iii)}] A parametrized circuit with some vanishing speed  consists of  an equilibrium circuit  $\{(e,0,T), (-e,0,S)\}$ for some $e\in \EE$ and $T,S>0$.
\end{itemize}
\end{Proposition}

\medskip

\subsection{Basic definitions}
 In this section we introduce a notion of tangent bundle $T \Gamma$ of $\Gamma$ and define  suitable sets of probability  measures  that we will use to build a  version of Mather theory on graphs.

\begin{Definition} \label{tgbundle}
The {\it tangent bundle} of $\Gamma$
is defined as
\begin{equation*}
T\Gamma:= \EE\times \R^+ / \sim,
\end{equation*}
where $\R^+:=[0,+\infty)$ and $\sim$ is the identification $(e,0)\sim (-e,0)$.\\
We denote each fiber by $\R^+_e := \{e\}\times \R^+$.
\end{Definition}

 We endow $T \Gamma:= \EE \times \R^+$ with a distance defined as: \begin{eqnarray*}
d((e_1,q_1), (e_2,q_2)) :=
\left\{
\begin{array}{ll}
q_1+q_2+1  & \qquad\hbox{if $e_1\neq \pm e_2$}\\
q_1+q_2  & \qquad\hbox{if $e_1=e_2$}\\
|q_1-q_2| & \qquad\hbox{if $e_1=-e_2$}.
  \end{array}
  \right.
\end{eqnarray*}
This makes $T \Gamma$ a Polish space. A set $A$ is open in  $T \Gamma$ in the induced topology if and only if $A \cap  \R^+_e$ is open in the natural topology of $\R^+$ for any $e$. Accordingly, $F$ is a Borelian set on $T \Gamma$ if and only if  $F \cap  \R^+_e$ is Borelian in $\R_e^+$ for any edge $e$. \\

\begin{Definition}\label{defsupportmeasure}
Given $\mu$  a  Borel  probability measure on $T \Gamma$,
we define the {\it support of $\mu$} as the set
$$\supp_\EE \mu= \{e \in \EE \mid \mu(\R^+_e) > 0 \}.\\$$
\end{Definition}

\bigskip

\begin{Proposition}
Any Borel probability measure in $T\Gamma$ can be decomposed as the convex combination of Borel probability measures in each fiber, namely
\begin{equation}\label{amuchina}
\mu (F) = \sum_{e\in \EE} \la_e \, \mu_e (F\cap \R^+_e) \qquad \mbox{for any Borelian set $F \subseteq T\Gamma$},
\end{equation}
where $\mu_e$ are Borel probability measures on $\R^+_e$ and $\lambda_e \geq 0$ such that $\sum_{e\in \EE} \lambda_e =1$. In particular, $\supp_\EE \mu= \{e \in \EE \mid \la_e\neq 0 \}.$
\end{Proposition}

\begin{proof}
 We distinguish two cases, according to whether $\mu(e,0) =0$ or $\mu(e,0) > 0$. In the first case, we set $\la_e:= (\mu(\R^+_e))$: if $\la_e =0$  ({\it i.e.}, $e \not\in \supp_\EE \mu$), then the choice of $\mu_e$ is irrelevant; otherwise  we define $\mu_e$ as the  restriction of $\mu$  on $\R^+_e$, normalized in order to be a probability measure.

  If $(\mu(e,0) >0$, then  $\mu_e$ is not uniquely determined since we have a degree of freedom in sharing the contribute of $\mu(e,0)=\mu(-e,0)$ between $e$ and $-e$. For, we introduce two positive constant $m_e$ and $m_{-e}$, such that  $m_e+m_{-e}=1$,  and denote by $\hat \mu_e$ the restriction of $\mu$ to $\R^+_e \setminus \{0\}$. Then, we define
 \begin{eqnarray*}
   \mu_e &:=& \frac 1{\hat \mu_e(\R^+_e) + m_e \mu(e,0)} \, \hat \mu_e + m_e \, \de(e,0) \\
   \la_e&:=& \hat \mu_e(\R^+_e) + m_e \mu(e,0),
 \end{eqnarray*}
 where $\delta(e,0)$ denotes Dirac delta at $(e,0)$.
\end{proof}

\bigskip

% We  write $\supp \mu_e$ to indicate the usual support of the measure $\mu_e$ in $[0,+ \infty)$.
Note that a Borel probability measure $\mu = \sum_{e\in \EE} \la_e \, \mu_e$ has finite first momentum if and only such property holds for any $\mu_e$, namely
\[ \int_0^{+\infty}   q \, d\mu_e < +\infty  \txt{for any $e \in \EE$.} \]
\smallskip
We denote by  $\mathbb  P$  the family of Borel probability measures on $T \Gamma$ with finite first momentum and we endow it  with the  (first) Wasserstein distance (see, for example, \cite{Villani}). The corresponding convergence of measures can be expressed in duality with continuous functions $F(e,q)$ on $T \Gamma$ possessing linear growth at infinity;  namely,  given a sequence $\{\mu_n\}_n$ and $\mu$ in $\bM$
\[\mu_n \to \mu \, \Longleftrightarrow \, \int F(e,q) \, d \mu_n \to \int F(e,q) \, d \mu_n\]
for any  function $F$ continuous in $T\G$ such that
\[|F(e,q)| \leq a_e \, q + b_e \qquad\hbox{for any $q \geq 0$ and suitable  $a_e, b_e \in \R$.}\]

\medskip

\subsection{Closed probability measures on $T\Gamma$}

Let us  observe that for any  $\om \in \Cf^1(\Gamma, \R)$, the function
\[(e,q) \longmapsto \langle \om, q \,e \rangle \]
is continuous  with linear growth on $T \Gamma$.  Given $\mu = \sum_e  \la_e \, \mu_e \in \mathbb P$, we consequently define
\begin{eqnarray}
  \int \om \, d \mu &:=&  \sum_{e\in \EE} \la_e \, \int_0^{+\infty}   \langle \om, q \,e \rangle  \,d\mu_e \label{rota}\\  &=&   \left \langle \om, \sum_{e\in \EE} \left [\la_e \, \int_0^{+\infty} q \, d\mu_e\right ]  \, e \right  \rangle . \nonumber
\end{eqnarray}
This associates to $\mu$ a 1--chain
\begin{equation}\label{defrho}
  \rho(\mu):= \sum_{e \in \EE} \left [ \la_e \, \int_0^{+\infty} q \, d\mu_e \right ] \, e \in \Cf_1(\Gamma,\R).
\end{equation}

\medskip

 \begin{Definition}
 We say that $\mu$ is a {\it closed measure} if
\[ \int df d\mu= 0 \txt{for any $f \in \Cf^0(\Gamma,\R)$.}\]
We set $\bM := \{ \mu \in \mathbb P\, : \, \mu \;\hbox{is closed}\}$
\end{Definition}

\smallskip

\begin{Remark}\label{golpetrump}
{\bf (i)} Given $\mu \in \mathbb P$, we have for any $ g \in \Cf^0(\Gamma,\R)$
\[\int dg \, d\mu =   \langle dg ,   \rho(\mu)   \rangle,   \]
hence
\[ \hbox{$\mu$ is closed} \; \Longleftrightarrow \; \partial \rho(\mu)=0 \; \Longleftrightarrow \; \rho \in H_1(\Gamma,\R),\]
namely $\rho(\mu)$ is a 1--cycle.  We call it {\it rotation vector} (or {\em Schartzman asymptotic cycle}) of $\mu$.  This should  be compared   with the corresponding classical definitions in Aubry--Mather theory (see \cite{ContrerasIturriaga}, \cite{Sorrentinobook}).

\noindent {\bf (ii)}  Given $\mu \in \mathbb M$ and  $ \omega \in \Cf^0(\Gamma,\R)$, it follows from the definition of closed measure and \eqref{rota} that
$$ \int \om \, d \mu = \langle [\omega], \rho(\mu) \rangle, $$
{\it i.e.}, it only depends on the cohomology class $[\omega]\in H^1(\Gamma, \R)$.\\

\end{Remark}

\smallskip

\begin{Proposition}\label{closedclosed} The subset $\bM \subset \mathbb P$ is    convex  and closed in the Wasserstein topology.
\end{Proposition}
\begin{proof} The convexity property is obvious. Let $\mu_n$ be a sequence of closed probability  measures converging in the Wasserstein sense to $\mu$. We consider $g \in \Cf^0(\Gamma, \R)$,  then associating  to $dg $ the continuous function  on $T \Gamma$
with linear growth $(e,q) \longmapsto \langle dg, q \,e \rangle$
and taking into account \eqref{rota}, we get
\[ \int dg \, d\mu_n \to \int dg \, d\mu.\]
This concludes the proof.
\end{proof}

\medskip

Let us define the map ${\rho}: \bM \longrightarrow  H_1(\G,\R)$ that to any closed probability measure $\mu$ associates its rotation vector $\rho(\mu)$ (see Remark \ref{golpetrump} {\bf(i)}).
One  proves the following properties.

\begin{Proposition} \label{proprho}
The map $\rho$ is continuous and affine (for convex combinations), {\it i.e.,} for every $\lambda\in [0,1]$ and $\mu_1, \mu_2 \in \bM$
$$
{\rho}\left(\la \mu_1 + (1-\la)\mu_2  \right) = \la {\rho}(\mu_1) + (1-\la) {\rho}(\mu_2).\\
$$
In particular, it is surjective.
\end{Proposition}
\begin{proof}
Let us first prove continuity.
If $\mu_n\to \mu$  in $\bM$  and $\omega$ is any   element  of $\Cf^1(\G,\R)$ with cohomology class $c$, then associating  to $\om $  the continuous function  on $T \Gamma$ with linear growth
$ (e,q) \longmapsto \langle \om, q \,e \rangle $
and taking into account \eqref{rota},
we have  that if $\mu_n$ converges to $\mu$ in the Wasserstein sense then
$$
\langle c, \rho(\mu_n) \rangle = \int \omega \,d\mu_n \longrightarrow \int \omega \, d\mu =  \langle c, \rho(\mu) \rangle.
$$
Since $c$ has been arbitrarily chosen in $H^1(\G,\R)$,  $\rho(\mu_n) {\longrightarrow} \rho(\mu)$ as $n\to +\infty$, which proves continuity.

The fact that the map $\rho$ is affine (under convex combination) is an immediate consequence of the definition  of the rotation vector.

Finally, let us prove surjectivity.
Let $h\in H_1(\G,\R)$ given by
$
h= \sum_{i=1}^N a_i e_i,
$
with $\partial (h) =0$; we can assume that $a_i>0$ (otherwise we substitute $e_i$ with $- e_i$).
Then,  it is sufficient to consider the measure
$\mu = \sum_{i=1}^N  \frac{1}{N}\de(e_i,N a_i)$   -- where $\delta(e,q)$ denotes Dirac delta at $(e,q)$ --
and use \eqref{defrho} to check that
$$
\rho(\mu)  = \sum_{i=1}^N  \frac{N a_i} {N} e_i = \sum_{i=1}^N  a_i  e_i = h.
$$
\end{proof}

\medskip

\subsection{Occupation measures}
Let us introduce the notion of {\it occupation measure}, which  can the thought as a measure  representation of  a parametrized path.

\begin{Definition}
   Given a parametrized  path
$\xi=(e_i,q_i,T_i)_{i=1}^M$, the associated  occupation  measure is
defined as
\begin{equation}\label{simon}
\mu_\xi:= \frac 1{T_\xi} \,  \sum_{i=1}^M T_i \, \de(e_i, q_i),
 \end{equation}
 where $T_\xi = \sum_{i=1}^M T_i$ and $\delta(e,q)$ denotes Dirac delta concentrated on the point $(e,q)$.
 \end{Definition}

\begin{Remark}
{\bf (i)}  Taking into account that an edge $e$ can be equal to $e_i$ for different values of the index $i$,  we see that   an occupation measure restricted to any edge is the convex combination of Dirac measures. \\
{\bf (ii)} For any $e \in \EE$, $\de(e,0)$  is a  closed  occupation measure corresponding to the equilibrium circuit based on $e$ with vanishing speed and any pair of positive numbers as time parametrization.
\end{Remark}

\medskip

\begin{Proposition}
Let $\mu_\xi$ be an occupation measure associated to a parametrized path $\xi=\{(e_i,q_i,T_i)\}_{i=1}^M$. Then,
$\mu_\xi$ is closed if and only if  $\xi$ is a parametrized cycle.
\end{Proposition}

\begin{proof}
{Let $g \in \Cf^0(\G,\R)$. Observe that for every $e\in \EE$
$$
\int dg\, d\delta(e,0) =0
$$
since we are integrating the function $\langle dg, qe\rangle$ with respect to $\delta(e,0)$.
The statement is trivial if all $q_i$ vanish (see Proposition \ref{capizzi}).
Let us assume that some $q_i\neq 0$; then,  recalling  Definition \ref{para} and Proposition \ref{capizzi}:
\begin{eqnarray*}
\int dg \, d\mu_\xi &=& \frac 1{T_\xi} \, \sum_{i=1}^M T_i \int dg \, d\delta(e_i, q_i)  =
 \frac 1{T_\xi} \, \sum_{i\mid q_i \neq 0} T_i \< dg, q_ie_i \ra \\ &=& \frac 1{T_\xi} \, \sum_{i\mid q_i \neq 0} \big ( g(\tt(e_i)) - g(\oo(e_i)) \big )
 =  \frac 1{T_\xi} \, \big ( g(\tt(\xi)) - g(\oo(\xi)) \big ).
\end{eqnarray*}
Therefore, $\mu_\xi$ is closed if and only if $g(\tt(\xi)) = g(\oo(\xi))$ for every $g\in \Cf^0(\G,\R)$, which is equivalent to $\tt(\xi) = \oo(\xi)$, {\it i.e.}, $\xi$ is a parametrized cycle.}
\end{proof}

\smallskip

\begin{Remark}\label{rotvectoroccupmeasure}
Given a parametrized cycle $\xi$, we have (see \eqref{cincinbin} for the definition of $[\xi]$)
\[ \rho(\mu_\xi)  = \frac 1{T_\xi} \, \sum_{i=1}^M e_i = \frac{[\xi]}{T_\xi}.\]
\end{Remark}

\medskip

We close this section with a density result.   This theorem is well known for measures on the tangent bundle of a manifold, a piece of folklore according to \cite{Bernard}.  We will not use it  in the rest of the paper, however  we include it for two  reasons: firstly, it somehow validates our previous definition of occupation measures, secondly because the proof, which follows the same lines of  \cite[Theorem 31]{Bernard}, is simple and illuminating, and represents a nice application of weak KAM theory on graphs to the analysis of closed probability measures.

\smallskip

\begin{Theorem}\label{ari} The set of closed occupation measures is dense in $\bM$.
\end{Theorem}

\smallskip

The proof is in Appendix \ref{B}.

\bigskip

\section{Mather's theory on graphs}\label{secMather}

Mather theory is about the minimization of the action functional
\[ \mu \longmapsto \int \LL^\om \, d\mu\]
on suitable subsets of closed probability measures. Results and definitions of this section are inspired by the corresponding ones in the classical Mather  theory, see \cite{ContrerasIturriaga}, \cite{Sorrentinobook}. We provide full details to make the text self--contained.

\smallskip

\subsection{Existence of minimizers}
We recall the main compactness criterion in the Wasserstein  space $\mathbb P$ (see, for example,  \cite{Villani}).
\begin{itemize}
  \item[--] A subset $\mathbb K \subset \mathbb P$ is relatively compact  if and only for any $\eps >0$ there exists a compact subset $K_\eps$ of $T\G$ such that
      \[\int_{K^c_\eps} q \, d\mu  < \eps \txt{for any $\mu \in \mathbb K$,}\]
      where $K^c_\eps$ stands for the complement of $K_\eps$ in $T \G$. \\
\end{itemize}

 From the superlinearity property of $\LL$, we derive the following property.

\begin{Proposition}\label{wasse} Given $a \in \R$, the set
\[\mathbb K_a:= \left \{ \mu \in \bM \mid \int \LL \, d\mu \leq a\right \}\]
is compact in $\bM$.
\end{Proposition}
\begin{proof}  Assume that $K_a \neq \emptyset$, otherwise there is nothing to prove. According to  the compactness criterion in the Wasserstein space $\bM$ and the definition of $T\G$, it is enough to prove that,  given $\eps >0$, there exists $M_\eps >0$ such that
\[ \int_{\R^+_e \cap (M_\eps,+\infty)} q \, d\mu < \eps \txt{for any $e \in \EE$, $\mu \in \mathbb K_a$.}\]
If this is not the case, we find  $\eps >0$, $e_0 \in \EE$, a sequence of positively diverging numbers $M_n$ and a sequence of measures
\[ \mu^{(n)} = \sum_{e\in \EE} \la^{(n)}_e\, \mu^{(n)}_e  \in \mathbb K_a\]
such that
\[\int_{M_n}^{+ \infty} q \, d\mu^{(n)}_{e_0}  \geq \eps \txt{for any $n$.}\]
Taking into account that $\LL(e_0,\cdot)$ is superlinear, we find another positively diverging sequence $h_n$ satisfying
\[\LL(e_0,q) \geq h_n \, q \txt{for $q \geq M_n$.}\]
Since edges are finitely many, we can find a constant $b$ such that
\begin{eqnarray*}
  a &\geq& \int \LL(e,q) \, d\mu^{(n)} \geq \int_{M_n}^{+\infty} \LL(e_0,q) \, d\mu^{(n)}_{e_0} + b \\
  &\geq& h_n \, \int_{M_n}^{+\infty} q \, d\mu^{(n)}_{e_0} + b \geq h_n \, \eps +b, \\
\end{eqnarray*}
which, as $n$ goes to $+\infty$, leads to a contradiction.
\end{proof}

As a consequence:

\begin{Corollary}\label{corwasse} The action functional
$\mu \longmapsto \int \LL \, d\mu$
is lower semicontinuous on $\bM$.
\end{Corollary}

This in turn implies:

\begin{Theorem}\label{act}  \hfill
\begin{itemize}
  \item[{\bf (i)}] The action functional admits minimum in $\bM$;
  \item[{\bf (ii)}] Given $h \in H_1(\G,\R)$, the action functional admits minimum in $\rho^{-1}(h)$.
\end{itemize}

\end{Theorem}
\begin{proof} Recall that a lower-semicontinuous function admits minimum on compact sets. Therefore, {(\bf i)} follows from Proposition \ref{wasse} and Corollary \ref{corwasse}.
Similarly,  {\bf (ii)} follows from Proposition \ref{wasse}, Corollary \ref{corwasse}, and the fact that $\rho^{-1}(h)$ is closed in $\bM$ (the map $\rho: \bM \to H_1(\G,\R)$ is continuous in force of  Proposition \ref{proprho}). \end{proof}

\medskip

\subsection{Mather's minimal average actions and Mather measures} \label{sec5.2}
We define {\it Mather's $\beta$--function} as:
\begin{eqnarray*}
\beta : H_1(\G,\R) &\longrightarrow& \R\\
h &\longmapsto& \min_{\mu \in \rho^{-1}(h)} \int \LL \, d\mu.
\end{eqnarray*}

The above minimum does exist in force of Theorem \ref{act} {\bf (ii)}.\\

 \begin{Definition} \label{cip}
 We say that a  measure $\mu \in \bM$ is a {\it Mather measure} {\it with homology} $h$ if $\int \LL\; \d\mu = \beta(h)$.
We denote the subset of these measures by $\bM^h$.\\
We define  the {\it Mather set of homology $h$} as
\begin{equation}
\widetilde{\MM}^h := \bigcup_{\mu \in \bM^h}  \supp \mu \subset T\G,
\end{equation}
where $\supp \mu$ denotes the support of $\mu$ in $T\Gamma$.
\end{Definition}
\bigskip

\noindent \underline{Properties of $\beta$}:
\begin{itemize}
\item $\beta$ is convex. In fact, let $h_1,h_2 \in H_1(\G,\R)$,  $\la \in [0,1]$ and let us consider $\mu_i \in \bM^{h_i}$ for $i=1,2$.
If follows from Proposition \ref{proprho} that
\[\rho(\la \mu_1 + (1-\la) \mu_2) = \la h_1 + (1-\la) h_2.\]
 Moreover, using
 the linearity of the integral and the definition of $\beta$, we obtain:
\begin{eqnarray*}
\beta(\la h_1 + (1-\la)h_2) &\leq&  \int  \LL \; d(\la \mu_1 + (1-\la) \mu_2) \\
& =& \la  \int \LL \;d\mu_1 + (1-\la) \int  \LL\; d\mu_2 \\
&=& \la \beta(h_1) + (1-\la) \beta(h_2).
\end{eqnarray*}
\item $\beta$ is superlinear. This could be proved directly by using the superlinearity of $\LL$; however, we  deduce it from the finiteness of its convex conjugate $\alpha$ (see \eqref{defalpha} and Remark \ref{stravinsky}).
\end{itemize}

\medskip

We consider the  convex conjugate of $\be$, that we shall call
{\it Mather's $\alpha$-function}:
\begin{eqnarray*}
\al : H^1(\G,\R) &\longrightarrow& \R\\
c &\longmapsto& \max_{h\in H_1(\G,\R)} \left( \langle c, h
\rangle - \beta(h) \right),
\end{eqnarray*}
where $\langle c, h \rangle$ denotes the pairing between
$H^1(\G,\R)$ and $H_1(\G,\R)$ defined in section \ref{subsecpairing}.
\\

One can also characterize $\alpha$ in a variational way, which shows that it is finite everywhere:
\begin{eqnarray}\label{defalpha}
\alpha(c) &=& \max_{h\in H_1(\G,\R)} \left( \langle c, h \rangle - \beta(h) \right) \\
&=& \max_{h\in H_1(\G,\R)} \left( \langle c, h \rangle -  \min_{\mu \in \rho^{-1}(h)} \int \LL \, d\mu \right)\nonumber\\
&=& - \min_{h\in H_1(\G,\R)} \left( \min_{\mu\in \rho^{-1}(h)} \left ( \int  \LL\; d\mu  -  \langle c, \rho(\mu) \rangle \right) \right) \nonumber\\
&=& - \min_{\mu\in \bM} \int \LL^\omega \; d\mu, \nonumber
\end{eqnarray}
where $\omega \in \Cf^1(\G,\R)$  has  cohomology class  $c$.   Due to the superlinearity of $\LL^\omega$, we see, arguing as in Proposition \ref{wasse}, that the sublevels of $\LL^\om$ are compact in the Wasserstein  topology, and consequently by Proposition \ref{closedclosed} the minimum in the above formula does exist. Therefore $\al$ is finite, convex with convex conjugate equal to $\be$.

\begin{Remark}\label{stravinsky}
The fact that $\al$ is finite, convex with convex conjugate equal to $\be$,  implies that  $\beta$  has superlinear growth. In fact,  a convex function  on finite dimensional vector spaces possess a  finite convex conjugate if and only if it has superlinear growth, see \cite{Rockafellar}. \\
\end{Remark}

\medskip

\begin{Definition}\label{ciop}
Given $c$ in $H^1(\G,\R)$ and $\om$ in the class $c$, we say that a  measure $\mu \in \bM$ is a {\it Mather measure} {\it with cohomology} $c$ if $\int \LL^\om \, d\mu = -\alpha(c)$  (observe that being $\mu$ closed, this notion does not depend on the choice of the representative $\omega$, but only on its cohomology class).
We denote the subset of these measures by $\bM_c$.\\
We define  the {\it Mather set of cohomology $c$} as
\begin{equation}
\widetilde{\MM}_c := \bigcup_{\mu \in \bM_c}  \supp \mu \subset T\Gamma,
\end{equation}
where $\supp \mu$ denotes the support of $\mu$ in $T\Gamma$.
\end{Definition}

\medskip

As  a consequence of  Proposition \ref{wasse}, we have
\begin{Proposition}\label{mathersetareconvex}
 For any  $h \in H_1(\G,\R)$, $c \in H^1(\G,\R)$, the sets  of Mather measures $\bM^h$, $\bM_c$ are compact, convex subsets of $\bM$.
\end{Proposition}

\smallskip

Next proposition  will help clarify the relation between the two notions of Mather measures in Definitions \ref{cip} and \ref{ciop}. To state it, recall that, like any convex
function on a finite-dimensional space, $\be$ admits a subdifferential at each point $h\in H_1(\G,\R)$, \ie we can find $c\in  H^1(\G,\R)$ such that
$
\be(h')-\be(h)\geq \langle c,h'-h\rangle
$
for any $h\in H_1(\G,\R)$.
We will denote by $\partial \be(h)$ the set of $c\in H^1(\G,\R)$ that
are subdifferentials of $\be$ at $h$.
Similarly, we will denote by $\partial \al(c)$ the set of subdifferentials of $\al$ at $c$.\\
Fenchel's duality implies an easy characterization of subdifferentials (see for example \cite[Proposition 3.3.3]{Sorrentinobook}):
\begin{equation}\label{subdifferential}
c\in \partial \be(h)  \quad \Longleftrightarrow \quad  h\in \partial \al(c)  \quad   \Longleftrightarrow \quad
\langle c,h\rangle=\alpha(c )+\be(h).
\end{equation}

\medskip
The next proposition can be proven as the corresponding ones in the classical Mather theory,  with obvious adaptations (we omit the proof, see for example \cite[Proposition 3.3.4]{Sorrentinobook}).
\begin{Proposition}\label{proprelationalphabeta} \hfill
\begin{itemize}
\item[{\bf(i)}] $\mu \in \bM$ is a Mather measure with homology $h$ if and only if \[  \mu \in \bM_c  \txt{for any $c \in \partial \be(h)$.}\]
\item[{\bf (ii)}] For every $c \in H^1(\G,\R)$
    \[ \partial \alpha(c) =\{\rho(\mu) \mid \mu \in \bM_c\}.\]
\end{itemize}
\end{Proposition}

\smallskip

\begin{Corollary}\label{inclusionMsets}
If $c\in \partial \beta(h)$, then $\widetilde{\MM}^h \subseteq \widetilde{\MM}_c$. In particular:
$$
\widetilde{\MM}_c = \bigcup_{h\in \partial \alpha(c)} \widetilde{\MM}^h.
$$
\end{Corollary}

\begin{Remark}
We will say that  $\mu$ is a Mather measure {\em tout court},  if it is a Mather measure for some cohomology $c$, or equivalently it is a Mather measure of homology $\rho(\mu)$.
\end{Remark}

\bigskip

\section{Properties  of Mather measures} \label{sec6}

\subsection{Structural properties and Mather's graph property}

 Exploiting the strict convexity of $\LL(e,\cdot)$, we can derive this first property of Mather measures, namely that they consist of a finite convex combinations of Dirac deltas, in particular each edge  appears at most once.
\begin{Proposition}\label{postmin}
The restriction of any  Mather measure to an edge of its support is concentrated on a point.
\end{Proposition}

\begin{proof}
Let $\mu = \sum_{e\in \EE} \la_e \, \mu_e$ be a Mather measure. We set
\[ \nu := \sum_{e \in \EE} \la_e \, \de \left(e,\int_0^{+\infty}  q \, d\mu_e \right ).\]
Thanks to the   convexity of $\LL(e,\cdot)$ for each $e\in \EE$,  we can  apply  Jensen inequality to  $\mu_e $  and  get
\begin{eqnarray*}
  \int \LL(e,q) \,d\mu &=& \sum_{e \in \EE} \la_e \, \int_0^{+\infty}   \LL (e,q)
\,d\mu_e \geq \sum_{e \in \EE} \la_e  \, \LL \left (e, \int_0^{+\infty} q \,
d\mu_e \right ) \\
   &=& \int \LL(e,q) \, d\nu.
\end{eqnarray*}
Observe that $\rho(\mu)= \rho(\nu)$; hence,  due to the strict convexity of $\LL(e,\cdot)$ for each $e\in \EE$ and the fact that $\mu$ is a Mather measure, we conclude that equality must prevail in the above formula, and this is possible if and only if $\mu=\nu$.
\end{proof}

\medskip

\smallskip

\begin{Proposition}\label{mathergraph2}
Let $c\in H^1(\G,\R)$ and $h\in  H_1(\G,\R)$.
\begin{itemize}
\item[{\bf (i)}] If $(f,q_1), (f, q_2) \in \widetilde{\mathcal M}_c$ (resp. $\widetilde{\mathcal M}^h$) for some $f\in \EE$, then $q_1=q_2$.
\item[ {\bf (ii)}] If $(f,q_1), (-f, q_2) \in \widetilde{\mathcal M}_c$ (resp. $\widetilde{\mathcal M}^h$) for some $f\in \EE$, then $q_1=q_2=0$ and $\alpha(c)=\min \alpha$.\\
\end{itemize}
\end{Proposition}

\begin{proof}
Since, by Corollary \ref{inclusionMsets}, $\widetilde \MM^h$ is contained in some  $\widetilde \MM_c$, then it suffices to prove the property for the latter.\\
Let $(f,q_1)$, $(f,q_2) \in \widetilde \MM_c$; then, by Proposition \ref{postmin} there are two Mather measures  $\mu= \sum_{e \in \EE} \la_e \, \mu_e$, $\nu= \sum_{e \in \EE} \tau_e \nu_e$ in $\bM_c$ such that
\[ \la_{f} >0,\; \tau_{f} >0 \qquad {\rm and} \qquad \mu_{f}= \de(f,q_1), \; \nu_{f}=\de (f,q_2).\]
Due to the convexity of $\bM_c$ (see Proposition \ref{mathersetareconvex}), we have that $\frac 12 \, \mu + \frac 12 \, \nu$ is in $\bM_c$,
and the restriction of it on $f$ is  a convex combination with positive coefficients of $\de(f,q_1)$ and $\de(f,q_2)$. We then derive, again from  Proposition \ref{postmin}, that $q_1=q_2$, which concludes the  proof of item {(\bf i)}.

We proceed by proving {\bf (ii)}.
Let $(f,q_1)$, $(-f,q_2)\in\widetilde \MM_c$; then, there exists $\mu \in \bM_c$ such that  $f,-f \in {\rm supp}_\EE\, \mu$; in fact, by Definition \ref{ciop}, there exist $\mu_1,\mu_2\in \bM_c$ such that $(f,q_1)\in \supp \mu_1$ and
$(-f,q_2)\in \supp \mu_2$, hence it suffices to consider $\mu= \frac{1}{2} \mu_1 + \frac{1}{2}\mu_2$, which still belongs to $\bM_c$ (due to convexity, see Proposition \ref{mathersetareconvex}).

Let us define
$$
\tilde{\mu} := \frac{1}{1-(\lambda_1 + \lambda_2)} \big( \mu - \lambda_1 \delta(f,q_1) - \lambda_2 \delta(-f,q_2)\big)
$$
with $\lambda_1, \lambda_2 \in (0,1)$, $q_1, q_2\geq 0$, so that  $\mu$ can be written as
$$
\mu =  \lambda_1 \delta(f, q_1) + \lambda_2 \de (-f, q_2) + (1-\lambda_1 -\lambda_2) \tilde{\mu}.
$$
Note that $\pm f \not \in \supp_\EE \tilde{\mu}$ because of Proposition \ref{postmin}. \\
Assume, without any loss of generality, that ${\lambda_1}{q_1} \geq {\lambda_2}{q_2}$ (otherwise, invert the roles of $f$ and $-f$) and
define
\begin{equation}\label{convexcombinationS}
\ov q:= \frac{\lambda_1q_1 - {\lambda_2q_2}}{\lambda_1 + \lambda_2} =
\frac{\lambda_1}{\lambda_1 + \lambda_2}  q_1 +
\frac{\lambda_2}{\lambda_1 + \lambda_2} (- q_2) \geq 0.
\end{equation}
Consider the new measure
$$
\nu := (\lambda_1 + \lambda_2) \;\delta (f, \ov q) + (1-(\lambda_1 + \lambda_2))  \tilde{\mu}.
$$
Clearly, $\nu$ is  a probability measure and it is also closed; in fact:
\begin{eqnarray*}
\rho\big( (\lambda_1+\lambda_2)\; \de(f, \ov q ) \big) &=&  \big({\lambda_1 + \lambda_2}\big){\ov q} \, f \;=\;
\big(
\lambda_1q_1 - \lambda_2q_2
\big)\, f \\
&=& \rho\big( \lambda_1 \delta(f, q_1) + \lambda_2 \de (-f, q_2) \big),
\end{eqnarray*}
hence, $\rho(\nu)=\rho(\mu)$  is a $1$-cycle, which implies that $\nu$ is closed (see Remark \ref{golpetrump} {\bf (i)}).

In order to get a contradiction, we want to prove that the action of $\nu$ is less than the action of $\mu$, thus contradicting minimality of $\mu$.
In fact:
\begin{eqnarray} \label{differenzaazione1}
\int \LL \,d\nu  -  \int \LL \,d\mu &=&
 ({\lambda_1+\lambda_2}) \, \LL(f,\ov q) -
\lambda_1  \, \LL(f, q_1) - \lambda_2\, \LL({-f}, q_2) \nonumber \\
&=&({\lambda_1+\lambda_2})\,  \left( {\LL}(f,\ov q) -
\frac{\lambda_1}{\lambda_1+\lambda_2}  \, {\LL}(f, q_1) - \frac{\lambda_1}{\lambda_1+\lambda_2}\, {\LL}({f}, -q_2)
\right)
\nonumber \\
&\leq&0, \end{eqnarray}
where in the last inequality we have used  the  convexity of ${\LL}(f,\cdot)$; taking into account that ${\LL}(f,\cdot)$ is in addition strictly convex, we see that a strict inequality prevails in \eqref{differenzaazione1}, leading  to a contradiction, unless
$$
\ov q=q_1 = -q_2 \qquad \Longleftrightarrow \qquad q_1=q_2=0.
$$
The property that $\alpha(c)=\min \alpha$ follows from the fact that $\delta(f,0)$ belongs to $\bM_c$, hence $0 \in \partial \alpha(c)$ (see Proposition \ref{proprelationalphabeta} (ii)). Being $\alpha$ convex implies that $\alpha(c)$ is the minimum of $\alpha$.
\end{proof}

We can now derive a central property that can be read as an instance of  the celebrated {\em Mather's graph theorem} (see \cite[Theorem 2]{Mather91})
in the graph setting\footnote {Ironically, the term {\it graph}  appearing twice in this  sentence, is used with two completely distinct meanings.}.\\
 To state it more precisely, let us introduce the {\it projection} $\pi_\EE: T\G \to \EE$ defined as
\[\pi_\EE(e,q) := \left \{\begin{array}{cc}
                   e & \txt{if $q >0$} \\
                   \{e,-e\} & \txt{if $q =0$.}
                 \end{array} \right . \]

                 \smallskip

\begin{Remark}
Observe that the projection $\pi_\EE$ that we have defined is multivalued at some points: this is needed in order to cope with the fact that the elements $(e,0)$, $(-e,0)$ are identified in $T\G$, for any $e \in \EE$.\\
Alternatively, one could consider
$\pi_{\EE^+}: T\G \to \EE^+$,
denoting the projection on a given orientation $\EE^+$ of the graph (namely, $\pi_{\EE^+}(\pm e, q) = e$ for any $e\in\EE^+$). In the light of Proposition \ref{mathergraph2}, the graph property in Corollary \ref{ironic} continues to hold with such a projection and all related results  can be suitably restated.\\
\end{Remark}

\begin{Corollary}\label{ironic} {\bf (Mather graph property)} \, The restriction of $\pi_\EE$ to $\widetilde \MM_c$ and $\widetilde \MM^h$ is injective for every $c \in H^1(\G,\R)$, $h \in H_1(\G,\R)$.
\end{Corollary}

\begin{proof}
 Since, by Corollary \ref{inclusionMsets}, $\widetilde \MM^h$ is contained in some  $\widetilde \MM_c$, then it suffices to prove the property for the latter.
 The result then follows from Proposition \ref{mathergraph2} {\bf (i)}.
\end{proof}

\medskip

\begin{Remark}\label{ironicbis}
It follows from Corollary \ref{ironic} that for any $c \in H^1(\G,\R)$
$$\big( \pi_{\EE}| \widetilde \MM_c \big)^{-1}:   \pi_{\EE}\big(\widetilde \MM_c\big) \longrightarrow \widetilde \MM_c
$$
is a well-defined map. In Section \ref{WKAM}  we will describe this function more explicitly (see Theorem \ref{proHJomega}).  \\
\end{Remark}

Next result is an important step in our analysis. It puts in relation, via Proposition \ref{postmin}, Mather and occupation measures.

\begin{Theorem}\label{premamather} A closed probability  measure, whose restriction on any edge is concentrated on a point, is a convex combination of occupation  measures based on circuits.
\end{Theorem}
\begin{proof}  Let
\begin{equation}\label{mamath0}
 \mu = \sum_{e\in \EE} \la_e \, \de(e,q_e)
\end{equation}
 with $\la_e\geq 0$ and $\sum \la_e =1$, be a measure as
indicated in the statement.  We first assume that $q_e \neq 0$ for any $e$.  We argue  by finite induction on  the cardinality  of $\supp_\EE \mu$ indicated
 by $|\supp_\EE\mu|$.   By taking the function which is equal to $1$ at a given vertex $x$
and $0$ elsewhere,  and exploiting that $\mu$ is closed, we deduce that
the relation
\begin{equation}\label{mamath1}
 \sum_{ e \in \EE_x} \la_e \, q_e = \sum_{ e \in - \EE_x } \la_e \,
q_e \qquad \forall\; x\in \VV.
\end{equation}
 If $|\supp_\EE \mu| =2$, set   $\supp_\EE \mu= \{e,f\}$. By applying \eqref{mamath1} to $x=\oo(e)$, $x= \tt(e)$, we realize  that   $(e,f)$   makes up a circuit and
 \[\la \, q_e= (1-\la) \, q_f \txt{for some $\la \in (0,1)$.}\]
   This implies that
  \[\la = \frac{q_f}{q_e+q_f}=  \frac 1{q_e} \, \frac{q_e \,q_f}{q_e+q_f}=
  \frac {1/q_e}{\frac 1{q_e} + \frac 1{q_f}} \quad \hbox{and} \quad 1-\la= \frac {1/q_f}{\frac 1{q_e} + \frac 1{q_f}}.\]
  This implies that $\mu$ is  the occupation measure corresponding to the parameterized circuit $((e,q_e,1/q_e), (f, q_f,,1/q_f))$.

 Let us  now assume the assertion true for measures with support of
cardinality less than a given $M$, and assume  $ |\supp_\EE \mu|=M \geq 3 $.  Starting by any edge $e \in \supp_\EE \mu$, we choose one of the edges $f \in \supp_\EE \mu$
  with \[ \tt(e)= \oo(f)\]
  and we call it $\pi_1(e)$.
 This choice is possible,  for any  initial $e$,  because of \eqref{mamath1}. We iterate the procedure starting from $\pi_1(e)$ to define $\pi_2(e)$. Taking again into account \eqref{mamath1}, we see that we can go on until we reach $\pi_k(e)$ with
 \[\tt(\pi_k(e))= \oo(\pi_h(e)) \qquad\hbox{for some $h \leq k$.}\]
 The edges
 \[ \{ \pi_h(e),\;\pi_{h+1}(e),\; \cdots,\; \pi_k (e)\}\]
 make up a circuit contained in  $\supp_\EE \mu$. We set  $M' = k+1-h$,
\[e_i=  \pi_{h+i-1}(e),  \quad \la_i=\la_{e_i} \quad q_i=q_{e_i} \qquad\hbox{for $i=1, \cdots, M'$}\]
and  consider the parametrized circuit $\xi= \left (e_i,q_i, 1/q_i \right )_{i=1}^{M'}$.  The  associated occupation   measure is
 \begin{equation}\label{mamath001}
  \mu_\xi=  \frac 1{T_\xi} \sum_{i=1}^{M'} \frac 1{q_i} \,
 \de(e_i,q_i),
 \end{equation}
where  $T_\xi= \left (\sum_{i=1}^{M'} \frac 1{q_i} \right )$.   We distinguish two cases:

\begin{itemize}
\item[--]  If $M=M'$ we show that $\mu=\mu_\xi$, which proves the claim.
In fact, in this case   for any vertex $x$ of the graph there is an
alternative: either  no edge in $\supp_\EE \mu$ is incident on it or
there are exactly two incident edges, one with $x$ as initial point
and the other with $x$ as terminal point. By applying
\eqref{mamath1} we deduce
\begin{equation}\label{mamath002}
\la_i \, q_i = \la_j \,q_j=:A  \txt{for any
$i, \, j \in\{1,\cdots,M'\}$.}
\end{equation}
 This implies  that $\la_i = \frac A{q_i}$ for
any $i$, and, since $\sum_i \la_i=1$ we obtain
\[A = \left ( \sum_i \frac 1{q_i} \right )^{-1} = \frac 1T_\xi.\]
By exploiting the above relation plus \eqref{mamath0}, \eqref{mamath001}, \eqref{mamath002}   we obtain
\begin{eqnarray*}
\mu_\xi &=& \frac{1}{T_\xi}\sum_{i=1}^{M} \frac 1{ q_i} \, \de(e_i, q_i) = \sum_{i=1}^{M} \frac 1{T_\xi \, q_i} \, \de(e_i, q_i) \\
&=& \sum_{i=1}^{M} \frac A{ q_i} \, \de(e_i, q_i)
= \sum_{i=1}^{M} \la_i \, \de(e_i, q_i) = \sum_{e\in \supp_\EE \mu} \lambda_e \delta(e, q_e) = \mu.
\end{eqnarray*}
\item[--] Let us assume now  that $M' < M$ and define
\[ \la =   T \, \min_i q_i \, \la_i.\]
Observe that
\[  \frac \la{T q_i} \leq \la_i\txt{for any $i\in
\{1, \cdots, M'\}$}\] and consequently
\[ \la = \la \, \sum_{i}\frac 1{T q_i} \leq
\sum_{i} \la_i < 1 ,\] where the rightmost strict inequality
comes from the fact that $M' <M$.  Let us define the following probability measure
\[ \nu = \frac 1{1- \la} \,   \left [\sum_{i=1}^{M'}  \left ( \la_i -  \la   \, \frac 1{T q_i} \right )  \, \de(e_i,
q_i) + \sum_{e \not\in \supp_\EE \mu_\xi} \la_e \,\de(e,q_e) \right ].\] This is  actually a probability measure since
\[
 \sum_{i}  \left ( \la_i -  \la  \, \frac 1{T q_i} \right
) + \sum_{e \not\in \supp_\EE \mu_\xi} \la_e =
 \sum_{e \in \supp_\EE \mu} \la_e - \la \, \frac 1T \, \sum_{i} \frac
 1{q_i} = 1 - \la. \]
Moreover
\begin{eqnarray}
&&\la \, \mu_\xi + (1- \la) \, \nu  \label{mamath3}\\
&=& \la \, \left  [ \frac 1T \, \sum_i \frac 1{q_i} \, \de(e_i, q_i)  \right ] +  \sum_i  \left ( \la_i -  \la   \, \frac 1{T q_i} \right )  \,
 \de(e_i,q_i) + \sum_{e \not\in \supp_\EE \mu_\xi} \la_e \,\de(e,q_e)
 \nonumber\\
 &=& \sum_{e \in \supp_\EE \mu} \la_e \, \de(e,q_e)= \mu. \nonumber
\end{eqnarray}
We see from \eqref{mamath3} that $\nu$ is closed since both $\mu$
and $\mu_\xi$ are closed. In addition, some of the coefficients
$\la_i -  \la \,  \frac 1{T q_i}$ must vanish by the very
definition of $\la$.  The support of $\nu$ has then cardinality less
than $M$, and by inductive assumption $\nu$ is the convex
combination of  occupation measures based on circuits. The same holds true for
$\mu$ in force of \eqref{mamath3}.
\end{itemize}
 Let us now discuss the case in which some of the $q_e$'s vanish. Let  $\mu$ be as in  \eqref{mamath0}  and define
 \[ E= \{ e \in \supp_\EE \mu \mid q_e >0\}, \quad F= \{ f \in \supp_\EE \mu \mid q_f =0\}, \quad \la_F= \sum_{f \in F} \la_f.\]
  If $E = \emptyset$, then $\mu = \de(e,0)$ for a suitable  $e \in \EE$ and  this measure is  supported by the equilibrium  circuit based on $e$, so that the assertion is proved. We then assume  that both $E$ and $F$ are nonempty.
 We consider  the probability measure
 \[\nu =   \, \sum_{e \in E} \frac{\la_e}{1- \la_F} \, \de(e,q_e)\]
   and derive
  \[\mu = (1 - \la_F) \, \nu + \sum_{f \in F}  \la_f \, \de(f,0).\]
  By the first part of the proof there exist occupation  measures  $\mu_{\xi_i}$ corresponding to circuits $\xi_i$ with
 \[ \nu = \sum_i \si_i \, \mu_{\xi_i} \qquad \si_i >0, \, \sum_i \si_i =1.\]
 Summing, up we have
 \[ \mu = (1 - \la_F) \sum_i \si_i \, \nu_i + \sum_{f \in F}  \la_f \, \de(f,0) .\]
  This concludes the proof.
\end{proof}

\medskip

\subsection{Irreducible  Mather measures}    A point in a convex set is called {\em extremal} if  it cannot be obtained as convex combination of two distinct elements of the set. \\

 A closed probability measure  is said to be  {\it irreducible}   if it is extremal in $\bM$.
\smallskip

\begin{Proposition}\label{panza} A Mather  measure  is irreducible  if and only if it  is an occupation measure corresponding to a parametrized circuit.
\end{Proposition}
\begin{proof}  Let $\mu$ be a Mather measure. If it is not an occupation measure  supported by a parametrized circuit,  then by Theorem \ref{premamather}  and Proposition \ref{postmin} it must the convex combination of distinct  occupation measures supported on parametrized circuits. This proves that it is not irreducible.

Conversely,  assume for the purpose of contradiction that $\mu$ is an occupation measure  supported on a parametrized circuit and that it is not irreducible. Hence, there exist $\mu_1 \neq \mu_2$ in $\bM$, $\la \in (0,1)$ such that
\[\mu= (1-\la) \, \mu_1 + \la \, \mu_2.\]
This implies  by Proposition \ref{proprho} that
\[\rho(\mu)= (1- \la) \, \rho(\mu_1) + \la \, \rho(\mu_2).\]
We  thus have
\begin{eqnarray*}
\beta(\rho(\mu))  &=& \int \LL \, d\mu   = (1 -\lambda) \,  \int \LL \, d\mu_1 + \lambda \, \int \LL \, d\mu_2 \\
&\geq& (1- \lambda) \, \beta(\rho(\mu_1)) + \lambda \, \beta(\rho(\mu_2),
\end{eqnarray*}
due to the convex character of $\be$, equality must prevail in the above formula, so that both $\mu_1$ and $\mu_2$ are Mather measures.
 Taking again into account  Theorem \ref{premamather} and Proposition \ref{postmin}, we find an occupation measure $\nu$ supported on a parametrized circuit with $\supp_\EE \nu$ proper subset of $\supp_\EE \mu$. This is in contrast with $\mu$ being supported on a circuit.
 \end{proof}
\smallskip

\begin{Remark}\label{rotvectorirreduciblemeasures} It follows from  Remark \ref{rotvectoroccupmeasure} and
Proposition \ref{panza}, that  the rotation vector of  an occupation  measure $\mu$ must have a  special form:
\begin{equation}\label{oggipiove}
 \lambda \sum_{e\in \EE^+} \tau_e  \, e,
\end{equation}
where $\lambda >0$, $\tau_e \in \{0,\pm 1\}$, and $\EE^+$ denotes an orientation of the graph.
\end{Remark}

\begin{Theorem} For any $c \in H^1(\G,\R)$,  the set of Mather measures $\bM_c$ is the convex hull of the irreducible Mather measures with cohomology $c$, which are finitely many.
\end{Theorem}
\begin{proof} We know from Proposition \ref{mathersetareconvex} that $\bM_c$ is a convex set. We claim  that $\mu \in \bM_c$ is irreducible  if and only if it is an extremal point of $\bM_c$. It is trivial that if it is irreducible then it is extremal in $\bM_c$. Conversely, let $\mu$ be extremal in $\bM_c$, and assume that there exist $\mu_1$, $\mu_2$ in $\bM$, $\la \in (0,1)$  with
\[\mu = (1-\la) \, \mu_1 + \la \, \mu_2.\]
If If $\om \in \Cf^1(\Gamma,\R)$ is of cohomology $c$, we have
\[- \al(c)= \int \LL^\om \, d\mu= (1-\la) \, \int \LL^\om \, d\mu_1 + \la \, \int \LL^\om \, d\mu_2\]
which implies, by the minimality property of $\al(\cdot)$ that both $\mu_1$ and $\mu_2$ are Mather measures of cohomology $c$, which is impossible. This proves the claim.

Let $\mu \in \bM_c$ then by  Theorem \ref{premamather} it is convex combination of occupation measures supported on parametrized circuits. Arguing as in the first part of the proof, we see that all the measures forming  the convex combination are in $\bM_c$, and consequently by Proposition \ref{panza} they are irreducible Mather measures in $\bM_c$. This shows that $\bM_c$ is the convex hull of its extremal points. These extremal measures are finitely many since -- by the graph property in Corollary \ref{ironic} --  a circuit identifies the Mather measures supported on it, if any,  and the set of circuits in $\G$ is finite.

\end{proof}

As shown in the previous result, any $\bM_c$, for $c\in H^1(\G,\R)$, contains some irreducible measure. The situation is rather different for the sets $\bM^h$. In fact, we know Remark \ref{rotvectorirreduciblemeasures}
that if $\bM^h$ contains irreducible Mather measures, then $h$ must be as in \eqref{oggipiove}; hence, not all $\bM^h$ do contain them.
We can get some information on which $\bM^h$'s contain irreducible Mather measures by looking at the extremal points of the epigraph of $\be$.  We  recall that   the epigraph of $\beta$ is given by
$$
{\rm epi}(\beta):= \{(h, t)\in H_1(\G,\R)\times \R:\; t\geq \beta(h)\}
$$

As in the classical ergodic theory, we have:

\begin{Proposition}\label{epigraph}  Let $h\in H_1(\G,\R)$. If $(h,\beta(h))$ is an extremal point of  ${\rm epi(\beta)}$, then there exist irreducible Mather measures of rotation vector $h$.
\end{Proposition}
\begin{proof}
 Let $\mu$ be a Mather measure with rotation vector $h$; then, according to Theorem \ref{premamather}
 \[ \mu= \sum_{i=1}^M \lambda_i \mu_i\] \ with $\lambda >0$, $\sum_i \la_i =1$ and $\mu_i$  occupation measures supported on parametrized circuits. Let us define
  \[ h_i=\rho(\mu_i) \txt{for any $i = 1,\ldots, M$.}\]
  We have
\begin{eqnarray*}
\beta \big (\sum_{i=1}^M h_i \big ) &=& \beta(h) = \int \LL \, d\mu  \\
&=& \sum_{i=1}^M \lambda_i \int \LL \, d\mu_i
\geq   \sum_{i=1}^M \lambda_i \beta(h_i) .
\end{eqnarray*}
Due to the convex character of $\be$, we see that equality must prevail in the above sequence of inequalities, so that all the  $\mu_i$'s must be Mather measures.  In addition, thanks to Proposition \ref{panza}, they are irreducible Mather measures.  We in addition have  that
$$
(h,\beta(h)) = \sum_{i=1}^M \lambda_i (h_i,\beta(h_i))
$$
Since $(h,\beta(h))$ is an extremal point of  ${\rm epi(\beta)}$,  we must necessarily have $h_i=h$ for any $i$. Hence, all the $\mu_i$'s are irreducible Mather measures with rotation vector $h$.
\end{proof}

\bigskip

\section{Weak KAM facts} \label{sec7}

We pause the exposition  of Aubry Mather theory on graphs, to recall some basic results of weak KAM theory that we will use in the following section.   Note that coercivity and convexity of the Hamiltonian are sufficient for these results to hold  true. All the material is taken from  \cite{SiconolfiSorrentino},   which contains a comprehensive treatment of the topic.\\

We  consider  a $1$--cochain $\om$ with cohomology class $c$, and the  family of discrete  Hamilton--Jacobi equations on $\G$
\begin{equation}\label{HJa} \tag{HJ$^\omega_a$}
 \max_{-e \in  \EE_x} \mathcal H^\om(e,\langle du,  e\rangle) =a \txt{for $x \in \VV$, $ a \in \R$}
\end{equation}
 which  can be equivalently written as
\[u(x)= \min_{-e \in \EE_x}  \big( u(\oo(e)) + \si^\om(e,a)\big).\]
A function $u: \VV \to \R$ is called {\it solution} if   equality in \eqref{HJa} holds for every vertex $x$. If instead the left hand--side is less than or equal to $a$, we say that  $u$ is a {\it subsolution} of \eqref{HJa}.\\

We set
\[a_0:= \max_{e \in \EE} a_e.\]

\begin{Remark}\label{due}
It is clear that equation \eqref{HJa} does not even make sense if $a < a_0$, because in this case the $a$--sublevels of $\HH(e,\cdot)$ are empty for some edge $e$.
\end{Remark}
\smallskip

Given a path $\xi=(e_i)_{i=1}^M$ in $\G$, we define  for $a \geq a_0$ (see \eqref{sigmaom})
\[\si^\om(\xi,a):= \sum_{i=1}^M \si^\om(e_i,a).\]
Note that this definition only depends on the concatenated edges making up $\xi$, no parametrization is involved. We sometimes refer to $\si^\om(\xi,a)$ as the {\em intrinsic length} of the path $\xi$ related to the Hamiltonian $\HH^\om$ and the level $a$.

\smallskip

 \begin{Proposition}\label{sottosola} \hfill
 \begin{itemize}
   \item[{\bf (i)}] Equation \eqref{HJa} admits subsolutions if and only if
   \[\si^\om(\xi,a) \geq 0 \txt{for any closed path $\xi$.}\]
   \item[{\bf (ii)}] A function $u: \VV \to \R$ is a subsolution of \eqref{HJa} if and only if
\[u(x) -u(y) \leq  \si^\om(\xi,a)   \txt{for any path $\xi$ with  $\oo(\xi)=y$,  $\tt(\xi)=x$.}\]
 \item[{\bf (iii)}] There is one and only one value of $a$, called {\em critical value} of $\HH^\om$,  for which the corresponding equation has solutions on the whole  $\Gamma$. It is given  by
\begin{equation}\label{defhameff}
  \min \{a \in \R:  \;\hbox{\eqref{HJa} admits subsolutions}\}.
  \end{equation}
 \end{itemize}
 \end{Proposition}

 For a proof of these claims see \cite[Propositions 6.5,  6.8 and Theorem 6.16]{SiconolfiSorrentino}

 \medskip

Clearly the Hamiltonian $\HH^\om$ is not invariant by change of representative in the class $c$,  however its critical value does not depend on the chosen representative, but only on the cohomology class $c$. If, in fact, we replace  $\om$ by $\om'= \om + dw$, for some $w \in \Cf^0(\G,\R)$,  then, given any (sub)solution $u$ to the equation associated to $\HH^\om$, the function $u-w$ will be a (sub)solution to the equation associated to $\HH^{\om'}$.\\

We can therefore define a function
 \[\widetilde \al :H^1(\Gamma,\R) \to \R\]
 associating  to any cohomology class the critical value of $\HH^\om$, as defined in \eqref{defhameff} (it only depends on the cohomology class of $\om$). We call {\it critical} the equation
\[\max_{e \in - \EE_x} \mathcal H^\om(e,\langle du,  e\rangle) =\widetilde\al(c)\]
and qualify  as critical its (sub)solutions. According to Remarks \ref{un} and \ref{due}
\[ \widetilde \al(c) \geq a_0 \txt{ for any $c \in H^1(\G,\R)$.}\]

\smallskip

\begin{Proposition}\label{sottosolacritica} Given $c \in H^1(\G,\R)$ and $\om$ of cohomology class $c$, the critical value $\widetilde \al(c)$ is characterized by the following properties:
\begin{itemize}
  \item[{\bf (i)}] $\si^\om(\xi,\widetilde\al(c))\geq 0$ for all cycles $\xi $ in $ \G$;
  \item[{\bf (ii)}] there exists a cycle $\zeta$ with $\si^\om(\zeta, \widetilde\al(c))=0$.
\end{itemize}
\end{Proposition}

For a proof of these claims see \cite[Lemma 6.7, Corollary 6.9, Proposition 6.15 and Theorem 6.16]{SiconolfiSorrentino}.\\

\smallskip
We define the Aubry sets as follows:
\[{\mathcal A}_c:= \{e \in \EE \mid \;\hbox{belonging to some
cycle with $\si^\om(\xi,\widetilde\al(c)) = 0$}\}.\]

\begin{Remark}
Given an arbitrary path $\xi$, the intrinsic length $\si^\om(\xi,\widetilde\al(c))$ is not invariant for the change of representative, however invariance is valid if $\xi$ is a cycle.  This is the reason why the Aubry set only depends on $c$ and not on the representative $\om$.  \\
\end{Remark}

We state in the next proposition a  relevant property of the Aubry sets (see \cite[Lemma 7.3]{SiconolfiSorrentino}).

\begin{Proposition}\label{solissima}  Let $c \in H^1(\G,\R)$ and $\om \in \Cf^1(\Gamma,\R)$ be of cohomology class $c$. Then,  any  subsolution $u$ of $\HH^\om=\widetilde \al(c)$ satisfies
  \[ \< du,e \ra = \si^\om(e,\widetilde\al(c)) \quad\hbox{and} \quad \HH^\om (e, \<du,e\ra)= \widetilde \al(c) \txt{for $e \in \A_c$}.\]
  Consequently, the differentials of all such subsolutions coincide on $e \in \A_c$.
\end{Proposition}

\smallskip

The value of  $du$ on the Aubry set $\A_c$ is clearly not  invariant for change of representative in $c$, however the element $\frac \partial{\partial p}\HH^\om(e,\langle du,e\rangle)$,  namely the element  characterized by the equality
{\small  \begin{equation}\label{donald1}
 \frac\partial{\partial p} \HH^\om (e,\langle du,e\rangle) \, \langle du, e\rangle =  \LL^\om(e,\frac\partial{\partial p}\HH^\om(e,\langle du,e\rangle) + \HH^\om (e,\langle du,e\rangle) \qquad \forall\; e\in\A_c
 \end{equation}
 }
possesses such an invariance, as made precise by the following result. \\

 \begin{Lemma}\label{donald}  Let $\om, \om' \in \Cf^1(\Gamma,\R)$  be in the same cohomology class $c$,  and let $u$, $v$ be subsolutions to
 {\rm (HJ$^\omega_{\widetilde \al(c)}$)}
 and
  {\rm (HJ$^{\omega'}_{\widetilde \al(c)}$)},
% $\HH^{\om}= \widetilde\al(c)$ and $\HH^{\om'}= \widetilde\al(c)$,
respectively; then
 \begin{equation}\label{donald2}
  \frac\partial{\partial p}\HH^\om(e,\langle du,e\rangle)= \frac\partial{\partial p} \HH^{\om'}(e,\langle dv,e\rangle)\txt {for any $e \in \A_c$.}
 \end{equation}
 \end{Lemma}

 \begin{proof}  We set
 \[q_e:= \frac\partial{\partial p}\HH^\om(e,\langle du,e\rangle) \txt{for $e \in \A_c$.}\]
 We have that $\om'= \om + dw$ for some $w \in \Cf^0(\G.\R)$, and consequently
 \[d v = du - dw.  \]
 Let $e \in \A_c$, then keeping in mind \eqref{donald1} we have
 \begin{eqnarray*}
  q_e \, \langle dv, e\rangle &=& q_e \, \langle du, e\rangle - q_e \, \langle dw, e\rangle \\
   &=& \LL^\om(e,q_e) + \HH^\om (e,\langle du,e\rangle) - q_e \, \langle dw, e\rangle\\
   &=& \LL(e,q_e) - q_e \langle \om, e \rangle + \HH (e,\langle du -dw + dw + \om,e\rangle ) - q_e \, \langle dw , e \rangle \\ &=& \LL^{\om'}(e,q_e) + \HH^{\om'}(e,dv).
 \end{eqnarray*}
 This proves \eqref{donald2}.
 \end{proof}

\smallskip

We denote by $\mathcal Q_c: \A_c \to \R$ the function
 \begin{equation}\label{defQc}
 e \longmapsto \frac \partial{\partial p}\HH^\om(e,\langle du,e\rangle).
 \end{equation}
 by the monotonicity properties  of $\HH^\om(e,\cdot)$, $\mathcal Q_c(e)$ is non--negative  for any $ e \in \A_c$.

\medskip

\section{Weak KAM and Aubry--Mather} \label{WKAM}
In this section we put in relation weak KAM theory and Aubry Mather theory on graphs.

\smallskip

\subsection{Mather's $\alpha$ function and critical value}
\begin{Theorem}\label{kamalaa} Given  $c \in H^1(\G,\R)$ and $\om \in \Cf^1(\Gamma,\R)$  of cohomology class $c$,  we have:
\begin{itemize}
  \item[{\bf (i)}] $\widetilde\al(c)$  and $ \al(c)$ coincide, {\it i.e.}, the critical value of $\HH^\om$ and the minimal action of Mather measures of cohomology class $c$ are the same;
  \item[{\bf (ii)}] {if an irreducible measure belongs to  $\bM_c$, then it  is  supported on a  circuit $\zeta$ such that $\si^\om(\zeta,\al(c))=0$;}
    \item[{\bf (iii)}] { if $\zeta=(e_i)_{i=1}^N$ is a circuit such that $\si^\om(\zeta,\al(c))=0$ and ${\mathcal Q}_c(e_i) \neq 0$ for all $i=1,\ldots, N$, then there exists an irreducible Mather measure supported on a parametrization of $\zeta$.
    }
\end{itemize}
\end{Theorem}

We remark that Item {\bf (iii)} in Proposition \ref{kamalaa} might not hold if ${\mathcal Q}_c$ vanishes on some of the edges forming the circuit $\zeta$ of vanishing intrinsic length; see also Remark \ref{aubrymather}.

\smallskip

\begin{proof} We  denote by  $u$  a subsolution to  {\rm (HJ$^\omega_{\widetilde \al(c)}$)}. Taking into account the definition of Lagrangian, we get  for any closed probability measure $\mu$
\[\int \LL^\om(e,q) \, d\mu \geq \int \big [ q \, \langle du,e\rangle - \HH^\om(e,\langle du,e\rangle )\big ]  \, d\mu = - \widetilde\al(c),\]
which shows that
\begin{equation}\label{kamalaa1}
 -\al(c) \geq -\widetilde\al(c).
\end{equation}
Let $\xi=(e_i)_{i=1}^M$ be a circuit with
\[\si^\om(\xi,\widetilde\al(c)) = \sum_i \si^\om(e_i,\widetilde\al(c))=0\]
so that $\xi$ is contained in $\A_c$.
We have by Proposition \ref{solissima} and \eqref{donald1} that
\[\widetilde \al(c)=  \HH^\om(e_i,\< du,e_i\ra) = \si^\om(e_i,\widetilde\al(c)) \, \mathcal Q_c(e_i) - \LL^\om (e_i,\mathcal Q_c(e_i)). \]
We first assume that $\mathcal Q_c(e_i) \neq 0$   for every $i$, then we get
\[\si^\om(e_i,\widetilde\al(c))  = \frac 1{\mathcal Q_c(e_i)} \, \big ( \widetilde \al(c) + \LL^\om (e_i,\mathcal Q_c(e_i)) \big ).\]
By summing over $i$, we further obtain
\begin{equation}\label{kamalaa2}
0= \sum_{i=1}^M \frac 1{\mathcal Q_c(e_i)} \, \LL^\om (e_i,\mathcal Q_c(e_i)) + \left (\sum_{i=1}^M  \frac1{\mathcal Q_c(e_i)} \right ) \widetilde\al(c).
\end{equation}
We denote by $\mu_\xi$ the occupation measure associated with the parametrized circuit $(e_i, \mathcal Q_c(e_i),1/\mathcal Q_c(e_i))_{i=1}^M$, and deduce from \eqref{kamalaa2}
\[\int \LL^\om \, d\mu_\xi= - \widetilde\al(c)\]
which together with \eqref{kamalaa1} proves the item {\bf (i)}, in the case $\mathcal Q_c(e_i) \neq 0$ for every $i$; {in particular, this also proves {\bf (iii)}}.\\
 If some $\mathcal Q_c(e_i)$ vanishes, then according to Proposition \ref{capizzi},  $\xi$ is an equilibrium circuit based on some edge $e$, namely  $\xi =((e,0,T),(-e,0,S))$ for some $T,S>0$. In this case we have
\[\widetilde\al(c)=a_0=a_e \]
and
\[\LL^\om(e,0) = \LL^\om(-e,0)= - a_e = - \widetilde\al(c).\]
The  occupation measure related to $\xi$  is $\de(e,0)$, and we get
\[\int \LL^\om \, d \de(e,0)= \LL^\om(e,0)= -\widetilde \al(c).\]
This ends the proof of item {\bf (i)}. 
{Let $\mu \in \bM_c$ be an irreducible Mather measure. Then, we distinguish two cases (see Proposition \ref{panza}):
\begin{itemize}
\item $\mu$ is the occupation measure supported on a parametrized cycle $(e_i,q_i,1/q_i)_{i=1}^M$, with $q_i\neq 0$ for all $i=1,\ldots,M$. Denoting by $T:= \sum_{i=1}^M \frac 1{q_i}$ and  $\zeta:=(e_i)_{i=1}^M$, we get:
\begin{eqnarray}
 - \al(c)&=&  \int \LL^\om \,d \mu  =\frac 1T \, \sum_{i=1}^M \frac 1{q_i} \, \LL^\om(e_i,q_i)  \nonumber \\ &\geq& \frac 1T \, \left [\si^\om(\zeta,\al(c)) - T
\, \al(c)  \right ]  \geq -\al(c), \label{kamalaa3}
\end{eqnarray}
which implies  $\si^\om(\zeta,\al(c))=0$. 
\item Otherwise, $\mu =\delta(e,0)$, for some $e\in \EE$; in this case 
 we must have $\al(c)= a_e$
and
 \[ \si^\om (e,\al(c)) + \si^\om (-e,\al(c))=0,\]
hence the thesis follows with  $\zeta=(e,-e)$.
\end{itemize}
This concludes the proof of {\bf (ii)}.
}

\end{proof}

We deduce:

\begin{Corollary}\label{kamalab}  Let  $c \in H^1(\G,\R)$, for any $(e,q) \in \widetilde M_c$ we have
\[ \LL^\om(e,q)= \si^\om(e,\al(c)) \, q - \al(c).\]
\end{Corollary}

Recalling the definition of the Aubry set $\A_c \subset \EE$, we further  derive:

\smallskip
\begin{Corollary}\label{gigella} Given $c \in H^1(\Gamma, \R)$, we have
\[ \pi_\EE \left ( \widetilde{\MM}_c \right ) =: \MM_c \subseteq \A_c.\]
{In particular, equality holds if $c$ is such that $\alpha(c) > \min \alpha$.}
\end{Corollary}

\medskip
\begin{Remark}\label{aubrymather} Note that in general $\MM_c$ might be a strict subset of $\A_c$. The reason is that we can find a  circuit with vanishing intrinsic length  not admitting a suitable  admissible parametrization, so that we do not find an occupation measure supported on it.
An example is given by a graph with two vertices, say $x$ and $y$, and two edges $e$, $f$ connecting them. We assume that  {$e\neq -f$ and that $\oo(e)= \tt(f)= x$ and $\tt(e)=\oo(f)=y$}. We consider the Hamiltonian defined as follows:
\[\HH(e,p) = \HH(-e,p)= p^2, \; \HH(f,p) =(p+1)^2 -1, \; \HH(-f,p)= (-p+1)^2  -1.\]
It is easy to check that $0$ is the critical value and the vanishing function is a solution of the corresponding  critical equation. We moreover have
\[\si(e,0)= \si(-e,0)=0, \;  \si(f,0)=0,\, \si(-f,0)=2.\]
We therefore see that $(e,-e)$ is an equilibrium circuit so that  $\de(e,0)$ is a Mather measure and $e$, $-e$ belong to the Mather set. We also have that the circuit $(e,f)$ has vanishing intrinsic length, so that $f$ belongs to the Aubry set, however, according to the definition of parametrized path, $(e,f)$ does not admit any admissible parametrization with vanishing speed on $e$, and $f$ does not belong to the Mather set.

\end{Remark}
\medskip

Next theorem refines the information provided in Corollary \ref{ironic} and Remark \ref{ironicbis}.

\begin{Theorem}\label{proHJomega} Given $c \in H^1(\Gamma,\R)$,
\[\widetilde{\MM}_c= \{(e, \mathcal Q_c(e))  \mid e \in \MM_c\}.\]
\end{Theorem}
\begin{proof} Let $\om$ be of cohomology $c$. We know from Proposition \ref{solissima}  that the differentials of all subsolutions $u$ to {\rm (HJ$^\omega_{\al(c)}$)}  coincide on $\A_c$ and satisfy
\begin{equation}\label{HJomega1}
  \<du   ,e\ra = \si^\om(e,\al(c)), \quad \HH^\om(e,\<du,e\ra)= \al(c).
\end{equation}
Let $\mu$ be an irreducible  occupation measure in $\bM_c$,  and assume that it corresponds to a parametrized circuit $\xi=(e_i,q_i,T_i)_{i=1}^M$. We derive  from Corollary \ref{kamalab} that
\[\LL^\om(e_i,q_i)= \si^\om(e_i,\al(c)) \, q_i -  \al(c) \txt{for  $i=1,\ldots, M$.}\]
This implies  by \eqref{HJomega1}
\[\LL^\om(e_i, q_i)  + \HH^\om(e,\<du,e\ra)= \<du  ,e\ra \, q_i,\]
which yields $q_i= \mathcal Q_c(e_i)$, for $i=1, \cdots, M$, in view of \eqref{donald1}.

\end{proof}

\medskip

\subsection{Minimizers of   Mather's $\al$ function} \label{singular}

\begin{Proposition}\label{prepoli}  The minimum of the function $\al$ is equal to $a_0$.
\end{Proposition}
\begin{proof} The function $\al$ admits minimum because of its coercive character. Assume $c$ to be a minimizer of $\al$ and denote by $\om\in\Cf^1(\Gamma,\R)$ a representative of the cohomology class $c$. Then  there exists $\mu \in \bM_c$ with $\rho(\mu)=0$
in view of Proposition \ref{proprelationalphabeta} {\bf (ii)}. Taking into account the definition of rotation vector, we derive that for some edge $f$, both   $f$ and  $-f$ belong to $\supp_\EE \mu$.  This implies by Proposition \ref{mathergraph2} {\bf (ii)} that $\mathcal Q_c(f)=\mathcal Q_c(-f) =0$ and $\al(c)=\min \alpha$.
Since  $\mathcal Q_c(f)=0$, then:
\[ \frac{\partial}{\partial p}\HH^\omega (f,\< du,f \ra)= {\mathcal Q}_c(f)= 0 ,\]
where $u$ is a subsolution to  {\rm (HJ$^\omega_{\al(c)}$)}.
 We deduce that $\< du,f \ra$ is a minimizer of $\HH^\om(f,\cdot)$ and consequently
$$
\al(c) = \HH^\om(f,\< du,f \ra)= a_f \leq a_0 \leq \min\alpha,
$$
which implies that $\alpha(c)=\min \alpha = a_0$.
\end{proof}

\smallskip

\begin{Corollary}\label{poli} An element $c \in H^1(\G, \R)$ is a minimizer of $\al$ if and only if the function $\mathcal Q_c$ vanishes at some $e \in {\mathcal M}_c$.
\end{Corollary}
\begin{proof}
The fact that if
$\mathcal Q_c(e)=0$  for some $e \in \MM_c$ then $c$ is a minimizer of $\al$, has been proved in Proposition \ref{prepoli}.\\
 Conversely, if $c$ is a minimizer of $\al$, then  $\al(c)=a_f$ for some $f \in \EE$, by  Proposition \ref{prepoli}. This implies that $f \in \MM_c$, moreover, if $u$ is a subsolution to
{\rm (HJ$^\omega_{a_f}$)}, where $\om$ is a representative of $c$, we get
\[\HH^\omega(f,\<du,f\ra)= a_f.\]
Taking into account that $a_f$ is the minimum of $\HH^\omega(f,\cdot)$, we finally have
\[ \mathcal Q_c(f)= \frac \partial{\partial p} \HH^\omega(f, \<du,f\ra)  =0.\]
\end{proof}

\bigskip

\appendix

\section{  From networks to graphs }\label{netto}

In this appendix, we describe how it is possible to develop Aubry-Mather theory on networks, by means of the discrete theory that we have developed on graphs.

Let us start by recalling the definition of network, as given in \cite{SiconolfiSorrentino}. We consider  a finite collection  $\EN$ of regular simple oriented
curves  in $\R^N$ parametrized over $[0,1]$. If $\ga \in \EN$, we denote by  $- \ga \in \EN$  the curve
\[- \ga(s)= \ga( 1 -s) \qquad\hbox{for $s \in [0,1]$,}\]
with the same support of $\ga$ and opposite orientation.   We further assume
\begin{equation}\label{netw}
    \ga((0,1)) \cap \ga'((0,1)) = \emptyset \qquad\hbox{whenever $\ga \neq
\pm \ga'$.}\\
\end{equation}

\smallskip

A {\it network} $\GG$   is a subset of $\R^N$ of the form
\[ \GG = \bigcup_{\ga \in \EN} \, \gamma([0,1]) \subset \R^N,\]
the curves in  $\EN$ are   called {\it arcs} of the network.

We call  {\it vertices} the initial  and terminal points of the arcs, and denote  by  $\VV$ the
sets of all such vertices.  We assume that the network  is finite and   connected, namely the number of arcs and vertices is finite and there is a finite concatenation of  arcs linking any pair of vertices.\\

\begin{Remark}\label{Riemannian}
This setting can be naturally extended to the case in which $\GG$
is embedded in a Riemannian manifold  $(M,g)$ (for example by means of Nash embedding theorem \cite{Nash}).
\end{Remark}

\smallskip
We can associate to any network $\GG$ a finite and connected
abstract graph $\G= (\VV, \EE)$  with the same vertices of the
network and edges corresponding to  the arcs.
More precisely, we consider an  abstract set $\EE$ with a bijection
\begin{equation}\label{defPsi}
  \Psi: \EE \longrightarrow \EN.
  \end{equation}
This induces  maps  $o : \EE \longrightarrow \VV$,
$-{\phantom{o}}: \EE \longrightarrow \EE $
 via
 \begin{eqnarray*}
   \oo(e) =  \Psi(e)(0)  \quad {\rm and} \quad
   - e = \Psi^{-1}(-{\Psi(e)}),
 \end{eqnarray*}
satisfying the properties in the definition of graph, see Section \ref{networks}. \\

\subsection{ Hamiltonians and Lagrangians on networks}  An Hamiltonian   on $\GG$ is a collection of Hamiltonians
\[H_\ga: [0,1] \times \R \to \R; \qquad (s,p) \mapsto H_\ga(s,p)\]
labeled by the arcs. We  assume  the compatibility conditions
\begin{equation}\label{ovgamma}
   H_{- \ga}(s,p)= H_{\ga}(1-s,-p) \qquad\hbox{for any $\ga \in \EN$.}
\end{equation}
As we will discuss with more detail hereafter, we can associate to the family $H_\ga$   an Hamiltonian $\HH(e, \cdot)$ on $\G$. Exploiting the results of \cite{SiconolfiSorrentino}, we see that  the corresponding Hamilton--Jacobi equations
\[H_\ga(s,(u \circ \ga)')= a  \qquad\hbox{on $(0,1)$ for $\ga \in \EN$,} \]
and
\[\max_{-e \in \EE_x} \mathcal H(e,\langle du,  e\rangle) =a \txt{for $x \in \VV$, $ a \in \R$}\]
are equivalent, in the sense that if $u: \GG \to \R$ is a (sub)solution of the former then its trace on$\VV$ solves the latter, and, conversely, any function $w: \VV \to \R$ solution of the latter can be uniquely extended on $\GG$ in such a way that the  extended function is solution of the former equation. In addition, in \cite{SiconolfiSorrentino}  we  developed in parallel weak KAM results for the two equations, proved that the two critical values coincide, define the corresponding Aubry sets, {\em etc.}...\\

The aim of this appendix to determine a set of rather natural  assumptions on the $H_\ga$'s such that the corresponding Hamiltonian on the graph $\G$ satisfies {\bf (H1)}, {\bf (H2)}. This will allow to take advantage of the output of this paper to provide an  Aubry--Mather theory on networks. \\

 We require the $H_\ga$'s to satisfy the following properties:
\begin{itemize}
    \item[{\bf (H1$^\prime$)}] $H_\ga$ is {\em continuous} in $(s,p)$, {\em differentiable} in $p$ for any fixed $s$, and such that the function
        \[(s,p) \mapsto \frac \partial{\partial p} H_\ga(s,p)\]
        is continuous;
    \item[{\bf (H1$^\prime$)}] $H_\ga$ is {\em superlinear} in $p$, uniformly in $[0,1]$, namely
    \begin{equation}\label{newborn1}
  \lim_{r \to + \infty} \min  \left \{ \frac{H_\ga(s,p)}p  \mid p > r,
  \, s \in [0,1] \right \} = + \infty;
\end{equation}
    \item[{\bf (H3$^\prime$)}]  $H_\ga$ is {\em strictly  convex} in $p$;
 \item[{\bf (H4$^\prime$)}]  the map  $s \longmapsto  \min_{p \in \R} H_\ga(s,p)$  is constant  in $[0,1]$, for  any given $\ga \in \EN$.\\
\end{itemize}

 We define $a_\ga= a_{-\ga}$ as the value of the constant function appearing in the assumption {\bf (H4')},
in other terms    the
sublevel of the Hamiltonian $H_\ga$ corresponding to $a_\ga$ is a singleton  for any $s$; we further denote by $p^\ga_s$ the minimizer of $H_\ga(s,\cdot)$. Therefore {\bf (H4')} reads
\[ H_\ga(s,p^\ga_s)= a_\ga  \txt{for any $s \in [0,1]$.}\]

\smallskip

\begin{Remark}
Actually condition {\bf (H4')} is required only for $\gamma \in \EN$ such that $a_\gamma = \max\{a_\lambda :\; \lambda \in \EN\}$. We refer to \cite[Remark 3.3]{SiconolfiSorrentino} for an explanation of the role of this condition.
\end{Remark}

\smallskip

We fix  $\ga \in \EN$,  $e\in \EE$ with $\ga = \Psi(e)$. The procedure to pass from  $H_\ga$  to $\HH(e, \cdot)$ consists in the following three steps:

\begin{itemize}
  \item[--] consider, for any $s$, the inverse, with respect to the composition, of $H_\ga(s,\cdot)$ in $[p^\ga_s,+\infty)$, denoted  by $\si^+_\ga(s,\cdot)$;
  \item[--] for any fixed $a \geq a_\gamma$, integrate $\si^+_\ga(\cdot,a)$ in $[0,1]$ obtaining $\si(e,a)$, where
\begin{eqnarray*}
  \si^+_\ga(s,a) & := &  \max \{ p \mid H_\ga(s,p)=a\} \\
  \si(e,a) &:=& \int_0^1 \si_\ga^+(s,a) \, ds;
\end{eqnarray*}
  \item[--]  define
  \begin{equation}\label{newborn0}
    \HH(e,p) := \left \{ \begin{array}{cc}
               \si^{-1}(e, p)& \quad\hbox{for $p \geq \si(e,a_\ga)$} \\
               \si^{-1}(-e, - p) & \quad\hbox{for $p \leq \si(e,a_\ga)$}
             \end{array} \right .,
  \end{equation}
  where the inverse is with respect the composition.
\end{itemize}

\smallskip

\smallskip

It is easy to see that if $H_\ga$ is independent of $s$,  then $H_\ga(\cdot)$ and $\HH(e,\cdot)$ coincide.\\

\begin{Proposition}\label{newborn}  If assumptions {\bf (H1$^\prime$)}--{\bf (H4$^\prime$)} hold, then  $\HH(e, \cdot): \R \to \R$ satisfies  {\bf (H1)}--{\bf (H2)}.
Moreover, $a_e=a_\gamma$ and $p_e=\sigma(e, a_\gamma)$, as defined in \eqref{ae}.

\end{Proposition}

We need a preliminary result.

\begin{Lemma} \label{newbornlem}  The function $a \longmapsto \si(e,a)$  from
$[a_\ga,+\infty)$  to $\R$ is:
\begin{itemize}
  \item[{\bf (i)}] continuous and strictly increasing;
  \item[{\bf (ii)}] strictly concave with $\lim_{a \to + \infty} \frac{\si(e,a)}a =0$;
  \item[{\bf (iii)}] differentiable in $(a_\ga, + \infty)$ with $\lim_{a \to a_\ga} \frac \partial{\partial a} \si(e,a)= + \infty$.
\end{itemize}
\end{Lemma}
\begin{proof}  The claimed continuity and monotonicity properties in item {\bf (i)} have been already proved in  \cite[Lemma 5.15]{SiconolfiSorrentino}.   Exploiting the  strict convexity assumption
on $H_\ga$, we deduce that, for any $s \in [0,1]$, $\la \in (0,1)$, $a$, $b$ in $[a_\ga,+\infty)$
\begin{eqnarray}
  H_\ga \left (s, \si^+_\ga( s, (1-\la)a +\la b)  \right )  \nonumber&=& (1-\la) \, a +\la \,  b \\ &=& (1-\la) \, H_\ga(s,\si^+_\ga( s,a))
+\la \,H_\ga(s, \si^+_\ga( s,b))  \label{newborn0}\\
   &>& H_\ga(s, (1-\la) \, \si^+_\ga(s,a)+\la\,
\si^+_\ga(s,b)).   \nonumber
\end{eqnarray}
Since $H_\ga(s,\cdot)$ is increasing in the interval $(p_s, +
\infty)$, the  inequality in \eqref{newborn0}  yields
\[\si^+_\ga( s,(1-\la)a +\la b) > (1-\la) \, \si^+_\ga(s,a)+\la\,
\si^+_\ga(s,b).\] By integrating the above relation over $[0,1]$, we
finally get
\[\si(e,(1-\la)a +\la b) >  (1-\la) \,\si(e,a)+\la\,
\si(e,b),\] which shows the strictly concave character of $\si(e,\cdot)$.\\
 To prove  the limit relation in {\bf (ii)}, we exploit
the uniform superlinearity assumption {\bf (H2$^\prime$)} on $H_\ga$.
Assume by contradiction that there is a sequence $a_n \to \infty$
and a positive $M$  such that
\[\lim_{n\rightarrow +\infty} \frac{\si (e,a_n)}{a_n} > M.\]
It follows from the definition of $\si(e, a_n)$ that there  exist, for any $n$,  $s_n \in [0,1]$, $p_n \in \R$ such that
\[  H_\ga(s_n,p_n)=a_n \quad\hbox{and} \quad \frac{ p_n }{a_n}> M.\]
Hence, we derive
\[ p_n \to + \infty   \quad\hbox{and}  \quad \frac{H_\ga(s_n,p_n)}{p_n} < \frac 1M,\]
which is in contrast with \eqref{newborn1}. We deduce from  {\bf(H1$^\prime$)} that the inverse function $a \mapsto \si^+_\ga(s,a)$ is differentiable in $(a_\ga,+\infty)$. By  differentiating  under the integral sign, we further get  that $a \mapsto \si(e,a)$ is differentiable in $(a_\ga,+\infty)$ and
\[ \frac \partial{\partial a} \si(e,a)= \int_0^1 \frac \partial{\partial a} \si^+_\ga(s,a) \, ds.\]
We denote by $\om(\cdot)$ a uniform continuity modulus of  $ (s,a) \mapsto \si^+_\ga(s,a)$ in $[0,1] \times [a_\ga, a_\ga +1]$ and of $(s,p) \mapsto \frac \partial{\partial p} H_\ga(s,p)$ in $K$ (see assumption {\bf (H1$^\prime$)}), where
\[ K= \{ (s,p) \mid s \in [0,1], \, p \in [p_s,+\infty),\, H_\ga(s,p) \leq a_\ga+1\}\]
is compact by the superlinearity assumption {\bf (H2$^\prime$)}. Then
\begin{equation}\label{odette}
 0 \leq \frac \partial {\partial p} H_\ga(s,p) \leq \om(p-p_s) \txt {for $(s,p) \in K$.}
\end{equation}
Observe that
\[
a=H_\gamma(s,\sigma^+_a(s,a)) \qquad \Longrightarrow \qquad
1=\frac{\partial}{\partial p}H_\gamma(s,\sigma^+_a(s,a)) \, \frac{\partial}{\partial a}\sigma^+_a(s,a).\]
This  and \eqref{odette} imply that for $a \in (a_e,a_e+1)$ we have
\begin{eqnarray*}
  \frac \partial{\partial a} \si(e,a) &=& \int_0^1 \frac 1{\frac \partial{\partial p} H_\ga(s,\si^+_\ga(s,a))} \, ds  \\
  &\geq& \int_0^1 \frac 1{\om( \si^+_\ga(s,a)-p_s)} \,ds \geq \frac 1{\om \circ \om(a - a_\ga)}.
\end{eqnarray*}
From this we derive item {\bf (iii)}, and conclude the proof.
\end{proof}

\smallskip

 \begin{proof}[{\bf Proof of Proposition \ref{newborn}:}] \; We derive from \eqref{newborn0} and Lemma \ref{newbornlem} that $\HH(e,\cdot)$ is continuous in $\R$ and differentiable in $\R \setminus  \{\si(e,a_\ga)\}$. Taking into account that
 \begin{eqnarray*}
    \frac \partial{\partial p} \HH(e,p) &=& \frac 1{\frac \partial{\partial a} \si(e,\si^{-1}(e,p))} \txt{for $p >  \si(e,a_\ga)$}\\
   \frac \partial{\partial p} \HH(e,p) &=& - \frac 1{\frac \partial{\partial a} \si(-e,\si^{-1}(-e,p))} \txt{for $p <  \si(e,a_\ga)$}
 \end{eqnarray*}
 we derive from Lemma \ref{newbornlem} {\bf (iii)} that
 \[\lim_{p \to \si(e,a_\ga)}  \frac \partial{\partial p} \HH(e,p)=0,\]
 which implies that $\HH(e, \cdot)$ is differentiable in $\si(e,a_\ga)$ with vanishing derivative.  Strict convexity is straight forward from the previous discussion.
Let us  prove {\bf (H2)}, namely that
 $\lim_{p \to \pm \infty} \frac{\HH(e,p)}{|p|}= + \infty. $\\
Recalling \eqref{newborn0} and using {\bf (ii)} in Lemma \ref{newbornlem}:
 \begin{eqnarray}
\lim_{p \to + \infty} \frac{\HH(e,p)}{p}
&=& \lim_{p \to + \infty} \frac{\sigma^{-1}(e,p)}{p} = \lim_{a \to + \infty} \frac{a}{\sigma(e,a)} = +\infty.
 \end{eqnarray}
Similarly for $p \rightarrow -\infty$, considering $\sigma(-e,a)$.\\
Easily follows that $a_e=a_\gamma$ and $p_e=\sigma(e, a_\gamma)$ (see \eqref{ae}).
 \end{proof}

For every $\gamma\in \EN$,    consider the Lagrangian associated to $H_\gamma$, namely its {\it convex conjugate}  $L_\gamma: [0,1]\times \R \longrightarrow \R$  defined as
\begin{equation}\label{Lgamma}
L_\gamma(s, q) := \sup_{p\in \R} \big( p \,q - H_\gamma(s, p) \big),
\end{equation}
where equality is achieved for $p$ such that $\frac{\partial H_\gamma}{\partial p}(s, p)=q$.

Since  $H_\gamma$ satisfies {\bf (H1$^\prime$)--(H3$^\prime$)}, then it follows (see for example \cite{Rockafellar}) that    $L_\gamma$ is
{\it continuous} in $(s,q)$, {\it differentiable}, {\it superlinear} and {\it strictly convex} in $q$.\\

Using \eqref{ovgamma} we see that the $L_\gamma$'s satisfy the following compatibility condition:
 \[L_{- \gamma}(s,q) = L_\gamma(1-s, - q) \qquad \forall\; s\in [0,1],\; q\in \R. \]

\medskip

\subsection{How to develop Aubry-Mather theory on  networks}\label{dragone}  In this section we look, from the point of view of networks, at some of the notions that we have introduced in the previous sections.
This will help clarify and validate the setting that we have proposed, and it will  outline the ideas and the tools that are needed in order to transfer the previous construction  to the network setting.\\

 In this section we assume the Hamiltonian $\{H_\ga\}_{\ga \in \EN}$ to be {\it Tonelli}, namely,  besides {\bf (H1$^\prime$)}--{\bf (H4$^\prime$)}, we further require that for any $\ga \in \EN$
 \begin{itemize}
   \item[{\bf (H5$^\prime$)}] $L_\ga(s,q)$ is of class $C^2$ in $(s,q)$  and $\frac{\partial^2}{\partial q^2}L_\gamma$ is positive definite.
 \end{itemize}

 \smallskip

We consider the network $\GG$ and its corresponding abstract graph $\G$. We fix an arc $\ga$ and an edge $e$ with $\Psi(e)=\ga$. \\

Given a parametrization $(q_e,T_e)$ of the edge $e \in \EE$, we provide an interpretation of it on the corresponding arc $\ga$. We first assume $q >0$, so that, according to the definition of parametrized path, $T_e= \frac {1}{q_e}$. Then, due to the strict convexity of $\LL(e,\cdot)$,  there exists a unique $p_{q_e} \geq p_e$ such that
\begin{equation}\label{draghi1}
 \LL(e,q)= p_{q_e} \, q_e -\HH(e,p_{q_e}) = q\, \si(e,a_{q_e})-a_{q_e},
\end{equation}
where $a_{q_e}>a_e$ is such that $p_{q_e}=\si(e,a_{q_e})$ (it is uniquely defined because of the continuity and  strict monotonicity of $\sigma(e, \cdot)$ stated in Lemma \ref{newbornlem}).
This  also implies the relation
\[q_e= \frac \partial{\partial p} \HH(e,p_{q_e}) =  \frac \partial{\partial p} \HH(e, \si(e,a_{q_e})).\]
We consider  the solution to $H_\ga (s, w'(s))=a_{q_e}$  in $(0,1)$ given by
\[w(s)= \int_0^s \si^+_\ga(t,a_{q_e}) \, dt,\]  and the  orbit of the Hamiltonian flow related to $H_\ga$ in $[0,1] \times \R$ with initial datum $(0,\si^+_\ga(0,a_{q_e})) = (0,w'(0))$,  contained in the energy level $a_{q_e}$. This orbit has as first component the curve $\xi_{q_e}$ with $\xi_{q_e}(0)=0$ and
 \[\dot \xi_{q_e}= \frac \partial{\partial p} H_\ga(\xi_{q_e}(t), w'(\xi_{q_e}(t))),\]
while the second component is given by $w'(\xi_{q_e}(t))$. We have in fact
\begin{eqnarray*}
0 &=& \frac d{dt} H_\ga(\xi_{q_e}(t),w'(\xi_{q_e}(t))) \\
   &=&  \frac {\partial}{\partial s} H_\ga(\xi_{q_e}(t), w'(\xi_{q_e}(t))) \, \dot\xi_{q_e}(t)+ \dot \xi_{q_e}(t) \, \frac d{dt} w'(\xi_{q_e}(t)),
\end{eqnarray*}
and accordingly
\[\frac d{dt} w'(\xi_{q_e}(t))= - \frac {\partial}{\partial s} H_\ga(\xi_{q_e}(t), w'(\xi_{q_e}(t))).\]
The orbit is defined in $[0,T_{q_e}]$, where $T_{q_e}$ is the time in which $\xi_{q_e}$ reaches the boundary point $s=1$.
\begin{Proposition}\label{draghi}
Let $q_e>0$ and let $\xi_{q_e}$ and $T_{q_e}$ be defined as above. Then:
\begin{enumerate}
  \item[{\bf (i)}] The time $T_{q_e}$ is equal to $\frac {1}{q_e}$;
  \item[{\bf(ii)}] $q_e$ is the average speed of $\xi_{q_e}$ in the time interval $[0,T_{q_e}]$;
  \item[{\bf (iii)}] $ \LL(e,q_e) =\frac 1{T_{q_e}} \int_0^{T_{q_e}} L_\ga(\xi_{q_e},\dot\xi_{q_e}) \, dt $;
    \item[{\bf (iv)}]  $\LL(e,q_e) = \frac 1{T_{q_e}} \min\left\{
 \int_0^{T_{q_e}} L_\ga(\zeta(t), \dot{\zeta}(t))\,dt  \right \},$
where the minimum is taken in the family of absolutely continuous curves $\zeta:[0, T_{q_e}]\longrightarrow [0,1]$  with $\zeta(0)=0$, $\zeta(T_{q_e})=1$.
\end{enumerate}
 \end{Proposition}
 \begin{proof} We have that $\dot\xi_{q_e}(t)$ and $w'(\xi_{q_e}(t))$ are conjugate in $[0,T_{q_e}]$, in the sense that
 \[\dot\xi_{q_e}(t)\, w'(\xi_{q_e}(t))= L_\ga(\xi_{q_e}(t),\dot\xi_{q_e}(t))+ H_\ga(\xi_{q_e}(t),w'(\xi_{q_e}(t))).\]
 which implies
 \begin{equation}\label{draghi2}
   L_\ga(\xi_{q_e}(t),\dot\xi_{q_e}(t))=  \dot\xi_{q_e}(t)\, \si^+_\ga(\xi_{q_e}(t),a_{q_e}) - a_{q_e}.
 \end{equation}
In addition, it follows from the definition of $L_\ga$ that
 \begin{equation}\label{draghi3}
L_\ga(\xi_{q_e}(t),\dot\xi_{q_e}(t)) \geq  \dot\xi_{q_e}(t))\, \si^+_\ga(\xi_{q_e}(t),b) - b   \quad\hbox{for any $b \geq a_e$}.
 \end{equation}
 By integrating  \eqref{draghi2}, \eqref{draghi3} over $[0,T_{q_e}]$ we further get
 \begin{eqnarray}
   \int_0^{T_{q_e}} L_\ga(\xi_{q_e}(t), \dot\xi_{q_e}(t))&=& \si(e,a_{q_e}) - T_{q_e} \, a_{q_e}  \label{draghi4}\\
  \int_0^{T_{q_e}} L_\ga(\xi_{q_e}(t), \dot\xi_{q_e}(t))&\geq& \si(e,b) - T_{q_e} \, b \quad\hbox{for any $b \geq a_e$}. \nonumber
 \end{eqnarray}
Taking into account \eqref{traviata}, we derive
\begin{equation}\label{draghi5}
 \LL(e,1/T_{q_e}) = \frac 1{T_{q_e}} \, \si(e,a_{q_e}) - a_{q_e}.
\end{equation}
 This implies by  \eqref{draghi1} and the strict convexity of $\LL(e,\cdot)$
 \[T_{q_e}= \frac 1{q_e} \qquad\hbox{and} \qquad q_e = \frac 1{T_{q_e}} \, \int_0^{T_{q_e}} \dot\xi_{q_e}(t) \, dt, \]
showing items {\bf (i)} and  {\bf (ii)}. By combining \eqref{draghi4} and \eqref{draghi5}, we get  {\bf (iii)}.\\
Finally, to obtain item {\bf (iv)}, it is enough to observe that for any absolutely continuous curve $\zeta$ in $[0,1]$ with $\zeta(0)=0$ and $\zeta(T_{q_e})=1$, one has
\[\int_0^{T_{q_e}} L_\ga(\zeta(t), \dot\zeta(t))\geq \si(e,b) - T_{q_e} \, b.\]
 \end{proof}

 \smallskip

 \begin{Remark} The equality in item {\bf (iii)} of Proposition \ref{draghi} can be interpreted by saying that the action functional on the graph computed in $\de(e,q)$ equals the action functional on the network computed in the occupation measure corresponding to the speed curve $(\xi_{q_e}(t),\dot\xi_{q_e}(t))$ in $[0,T_{q_e}]$. The latter measure is obtained by pushing forward through $(\xi_{q_e}(t),\dot\xi_{q_e}(t))$ the $1$--dimensional Lebesgue measure restricted to $[0,T_{q_e}]$ and normalize it.

{In particular, item {\bf (iv)} of Proposition \ref{draghi} reads that the curve $\xi_{q_e}$ defined on $[0,T_{q_e}]$ is  action minimizing for $L_\gamma$.  This sheds light on the reason why Mather measures on the graph consist of convex combinations of Dirac deltas (see Proposition \ref{postmin}), in analogy with what happens in the classical theory, in which Mather measures are supported on action-minimizing curves (see \cite{Mather91, Sorrentinobook}).}
 \end{Remark}

\begin{Remark} In the  case where $e \in \MM_0$ and $q_e= \mathcal Q_0(e) >0$ -- we have chosen the cohomology class $0$ just for simplicity --,  the above construction acquires a global significance, in the sense that $\si^+_\ga(s,\al(0))$  is not just the derivative of a local (in $(0,1)$) solution of $H_\ga=\al(0)$, but we also have that
\[\si^+_\ga(s,\al(0)) = \frac d{ds} u\circ \ga(s)\]
for any critical subsolution $u$  of the Hamilton--Jacobi equation on the network, see  \cite{SiconolfiSorrentino}. \\
\end{Remark}

\medskip

\begin{Remark}
To discuss the case when the speed $q_e$ vanishes  for some $e \in \EE$,  the equilibrium circuit $(e,-e)$ with the parametrization $((e,0,T_1),(-e,0,T_2))$, with $T_1$, $T_2$ positive constants. We set $\ga=\Psi(e)$ and consequently $- \ga=\Psi(-e)$.  We have
\[\LL(e,0)= \LL(-e,0) = - \HH(e,p_e)= \HH(-e,p_{-e})= -a_e=a_{-e}.\]
In addition we have by assumption {\bf (H4$^\prime$)}
\[L_\ga(s,0)= L_{-\ga}(s,0)= a_e= a_{-e} \qquad\hbox{for every $s \in [0,1]$}\]
so that
\begin{eqnarray*}
  0&=&\frac \partial{\partial s}L_\ga(s,0)=- \frac \partial{\partial s}H_\ga(s,\si^+_\ga(s,a_e)) \\
  0&=&\frac \partial{\partial s}L_{-\ga}(s,0)=- \frac \partial{\partial s}H_{-\ga}(s,\si^+_{-\ga}(s,a_{-e})).
\end{eqnarray*}
This implies that all the points  $(s,\si^+_{a_e}(\ga,s))$, $(s,\si^+_{a_{-e}}(-\ga,s))$ are equilibria of the Hamiltonian flows related to $H_\ga$, $H_{-\ga}$, respectively.  We can put in relation the measures $\de(e,0)=\de(-e,0)$ with the Dirac measures concentrated at all points of the arcs $\ga$,$-\ga$, which -- in analogy with what we did in the graph -- can be identified.
\end{Remark}

\bigskip

\section{Proof of Theorem \ref{ari}}\label{B}

We need a preliminary result, see \cite[Proposition 42]{Bernard}.

\begin{Lemma}\label{younga} Let $\mathbb K$ be a closed convex subset of $\mathbb P$, we set
\[C^+ =\left\{ F: T\G \to \R \;\hbox{continuous with linear growth} \mid \int F \, d\mu \geq 0 \; \forall \, \mu \in \mathbb K \right \}.\]
Then:
\[\mathbb  K = \left\{ \nu \in \mathbb P \mid \int F \, d\nu \geq 0 \; \forall \,F \in C^+ \right \}.\]
\end{Lemma}
The proof of this lemma is based on  a separation result in Wasserstein spaces  that we take from \cite{LOST}.  We state it below with slight changes to adapt it to our  notation and setting.

\begin{Lemma}\label{preyounga}\cite[Theorem 2.9]{LOST} Let $\mathbb K$ be a closed convex subset of $\mathbb P$, and $\nu \not \in \mathbb K$. Then, there exists $F: T\G \to \R$ with linear growth such that
\[\int F \, d\mu > \int F \, d\nu \txt{for any $\mu \in \mathbb K$.}\]
\end{Lemma}
\smallskip

\begin{proof}[\bf{Proof of Lemma \ref{younga}:}] \; Given $\nu \not\in \mathbb K$, we fix $\mu_0 \in \mathbb K$ and define
\[\nu_\la= (1-\la) \, \nu + \la \, \mu_0 \txt{for $\la \in [0,1]$.}\]
Since $\mathbb K$ is closed, there exists $\la_0 \in (0,1)$ with $\nu_{\la_0} \not \in \mathbb K$. We denote by $F$ a  function satisfying  the statement of Lemma \ref{preyounga} with respect to $\nu_{\la_0}$; we can in addition assume, without loosing generality, that
\begin{equation}\label{younga01}
 \int F \, d\nu_{\la_0}=0.
\end{equation}
Therefore  $F \in C^+$ and
\begin{equation}\label{younga02}
 \int F \, d\mu > 0 \txt{for any $\mu \in \mathbb K$}
\end{equation}
It follows from the definition of $\nu_{\la_0}$, \eqref{younga01}, \eqref{younga02}  that
\[\int F \, d\nu <0.\]
Summing up, we have found that for any $\nu \not \in \mathbb K$, there exists $F \in C^+$ whose  integral with respect to $\nu$ is strictly negative. This proves the assertion.
\end{proof}

\smallskip

\begin{Lemma} The closure in $\mathbb P$ of the space of closed occupation measures is convex.
\end{Lemma}

The fact that the closure of the space of closed occupation measures is convex,  stems from the property  that a closed occupation measure stays unchanged under any finite repetition of the corresponding cycle. Therefore, we can connect a finite number of cycles through simple paths in order to make a unique cycle. We can then  repeat $n$ times the cycles leaving unaffected the connecting paths and  obtain a sequence of closed occupation  measures indexed by $n$ converging, as $n \to + \infty$, to a measure which does not ``{see}'' the connecting simple paths and is a convex combination of the occupation measures corresponding to the cycles with repetitions.\\
A formal argument can be found \cite[Lemma 30]{Bernard} for  measures on the tangent bundle of a compact manifold. It can be adapted with minor modifications to our setting.\\

\smallskip

We can now prove the main result of this appendix.

\begin{proof} [{\bf Proof of Theorem \ref{ari}:}]
In view of Lemma \ref{younga}, it is enough to show that if a continuous function $F$  with linear growth in $T\G$ satisfies
\begin{equation}\label{ari1}
 \int F \, d\mu \geq 0  \txt{for any closed occupation measure $\mu$,}
\end{equation}
then it also satisfies
\begin{equation}\label{ari2}
 \int F \, d\nu \geq 0  \txt{for any  measure $\nu \in \bM$.}
\end{equation}
Let $F$ satisfy  \eqref{ari1}. By integration  with respect to the closed occupation measures $\de(e,0)$, for any $e\in\EE$, we get
\[F(e,0) \geq 0 \txt{for any $e \in \EE$.}\]
Thanks to the above inequality, we can modify $F$ in $[0,1/n] \cup [n,+\infty) \subset \R^+_e$, for any $e\in \EE$,  constructing a sequence of  continuous  functions $F_n$  defined on $T\G$  such that
\begin{equation}\label{younga1}
  F_n(e,0) > 0,  \quad F_n(e,\cdot) \; \hbox{has superlinear growth at $+ \infty$} \txt{for any $e \in \EE$}
\end{equation}
and in such a way that  for any $n$, for each $e \in \EE$, $q \geq 0$
\begin{eqnarray}
  F_{n+1}(e,q) & \leq& F_n(e,q) \label{younga2}\\
  F_n(e,q) &\geq & F(e,q) \label{younga21} \\
  F_n(e,q) & \to & F(e,q) \txt{as $n \to + \infty$.} \label{younga3}
\end{eqnarray}
We define
\[G_n(e,p) := \max_{q \geq 0} \big( p \, q -F_n(e,q)\big);\]
the function $G_n(e,\cdot)$ is finite by the superlinear growth of $F_n$, convex and superlinear at $+ \infty$; in addition $G_n(e,\cdot)$ is increasing in $p$ and by \eqref{younga1}
\[\inf_{p\in \R} G_n(e,p)= \lim_{p \to - \infty}G_n(e,p)= - F_n(e,0) <0. \]
Therefore, the value $0$  is attained  by $G_n(e,\cdot)$ and is above the infimum.  We denote by $\varphi^n_e$, for any $e\in \EE$,  the unique element such that
\[G_n(e,\varphi^n_e)=0.\]
The quantity  $\varphi^n_e$   must be understood as an intrinsic length of the edge $e$ related to the Hamiltonian $G_n$ and the value $0$.  We have
\begin{equation}\label{younga32}
 \varphi^n_{e} \, q \leq F_n(e,q) \txt{for any $q \geq 0$}
\end{equation}
and there exists   $q_e >0$  with
\begin{equation}\label{younga31}
  0= G_n(e,\varphi_{e}) = \varphi^n_{e} \, q_e- F_n(e,q_e).
\end{equation}
We consider the discrete  Hamilton--Jacobi equation on $\G$
\begin{equation}\label{younga4}
 \max_{-e \in  \EE_x} G_n(e,\langle du,  e\rangle) =0 \txt{for $x \in \VV$.}
\end{equation}
We know from Proposition \ref{sottosola} {\bf (i)} that in  order \eqref{younga4} to have subsolutions it is necessary and sufficient that for any cycle $\xi=(e_i)_{i=1}^M$ in $\G$ the intrinsic length
\[\varphi^n(\xi) := \sum_{i=1}^M \varphi^n_{e_i} \geq 0.\]
We deduce from  \eqref{younga31} that
\begin{equation}\label{younga5}
  \varphi^n_{e_i} = \frac 1{q_i} \, F_n(e_i,q_i),
\end{equation}
where $q_i:=q_{e_i}$ (see \eqref{younga32}, \eqref{younga31}). We consider the parametrized version of $\xi$ given by $(e_i,q_i,1/q_i)_{i=1}^M$ and denote by $\mu_\xi$ the corresponding closed occupation measure. We have  by \eqref{younga21} and the assumption that
\[0 \leq \int F_n \, d \mu_\xi= \frac 1{\sum_{i=1}^M \frac 1{q_i}} \, \sum_{i=1}^M \frac 1{q_i} \, F_n(e_i,q_i),\]
which finally implies, using  \eqref{younga5}, that $\varphi^n(\xi) \geq 0$.
If $u: \VV \to \R$ is a subsolution  of \eqref{younga4}, we have
\begin{equation}\label{younga10}
 \langle du, e \rangle \leq \varphi^n_e \txt{for any $e \in \EE$.}
\end{equation}
Let $\nu = \sum_{e\in \EE} \la_e \, \nu_e$ be a closed  measure on $T\G$,  then by \eqref{younga32}, \eqref{younga10}
\begin{eqnarray*}
  0 &=& \sum_{e\in \EE} \la_e \, \int q \, \langle du, e \rangle  d \nu_e    \\
   &\leq& \sum_{e\in \EE} \la_e \, \int \varphi^n_e \, q  \, d\nu_e \\
   &\leq&  \sum_{e\in \EE} \la_e \, \int  F_n(e,q) \, d \nu_e \; =\; \int F_n \, d\nu.
\end{eqnarray*}
Taking into account \eqref{younga2}, \eqref{younga3} and passing to the limit as $n \to + \infty$ in the above inequality, we obtain
\[\int F \, d\nu \geq 0,\]
which shows that $F$ satisfies \eqref{ari2}. This concludes the proof.
\end{proof}

\color{black}

\bigskip

%%%%%%%%%%%%% BIBLIOGRAPHY %%%%%


\vspace{10 pt}

\begin{thebibliography}{10}
\expandafter\ifx\csname natexlab\endcsname\relax\def\natexlab#1{#1}\fi
\expandafter\ifx\csname bibnamefont\endcsname\relax
  \def\bibnamefont#1{#1}\fi
\expandafter\ifx\csname bibfnamefont\endcsname\relax
  \def\bibfnamefont#1{#1}\fi
\expandafter\ifx\csname citenamefont\endcsname\relax
  \def\citenamefont#1{#1}\fi
\expandafter\ifx\csname url\endcsname\relax
  \def\url#1{\texttt{#1}}\fi
\expandafter\ifx\csname urlprefix\endcsname\relax\def\urlprefix{URL }\fi
\providecommand{\bibinfo}[2]{#2}
\providecommand{\eprint}[2][]{\url{#2}}


\bibitem{Achdou} Y. Achdou, M. Dao, O. Ley, N. Tchou.
\newblock A class of infinite horizon mean field games on networks.
\newblock {\em Networks Heterog. Media} 14 (3): 537--566, 2019.


\bibitem{Aubry}
Serge Aubry and P. Y. Le Daeron.
\newblock {The discrete {F}renkel-{K}ontorova model and its extensions.
              {I}. {E}xact results for the ground-states}.
\newblock {\em Phys. D}, 8 (3): 381--422, 1983.




\bibitem{Bernard}
Patrick Bernard.
\newblock Young measure, superposition and transport.
\newblock{\em  Indiana University Math Journal} 57(1): 247--276, 2008.

\bibitem{BB1} Patrick Bernard and Boris Buffoni.
\newblock The Monge problem for supercritical Ma\~n\'e potentials on compact manifolds.
\newblock {\em Adv. Math.} 207: 691--706, 2006.


\bibitem{BB2}   Patrick Bernard and Boris Buffoni.
\newblock Optimal mass transportation and Mather theory.
\newblock {\em J Eur. Math. Soc.} 9 (1): 85--121, 2007.

\bibitem{CAMI}  Fabio Camilli and  Claudio Marchi.
\newblock Stationary Mean Field Games systems defined on networks.
\newblock {\em SIAM J. Control Optim.} 54: 1085--1103, 2016.

\bibitem{CAMI2} Fabio Camilli, Raul De Maio and Andrea Tosin.
\newblock Transport of measures on networks.
\newblock {\em Networks and Heterogeneous Media} 12 (2): 191--215,  2017.



\bibitem{ContrerasIturriaga}
Gonzalo Contreras and Renato Iturriaga.
\newblock Global minimizers of autonomous Lagrangians.
\newblock {\em 22o Col\'oquio Brasileiro de Matem\'atica}, Instituto de Matem\'atica Pura e Aplicada (IMPA), Rio de Janeiro, 148 pp., 1999.

\bibitem{CIS} Gonzalo Contreras, Renato Iturriaga and Antonio Siconolfi.
\newblock Homogenization on arbitrary manifolds.
\newblock {\em Calc. of Variations and PDE} 52  (1--2):   237--252, 2015.



\bibitem{Fathi}
Albert Fathi.
\newblock Th\'eor\`eme KAM faible et th\'eorie de Mather sur les syst\`emes lagrangiens.
\newblock {\em C. R. Acad. Sci. Paris S\'er. I Math.} 324 (9): 1043--1046, 1997.

\bibitem{FathiICM}
Albert Fathi.
\newblock Weak KAM theory: the connection between Aubry-Mather theory and viscosity solutions of the Hamilton-Jacobi equation.
\newblock {\em Proceedings of the International Congress of Mathematicians -- Seoul 2014.} Vol. III:  597--62, 2014.



\bibitem{FS}  Nicolas Forcadel and Wilfredo Salazar.
\newblock  Homogenization of a discrete model for a bifurcation and application to traffic flow.
\newblock {J. Math. Pures Appl.}  136 (9): 356--414, 2020.

\bibitem{GIM} Giulio Galise, Cyril  Imbert, and R\'{e}gis Monneau.
\newblock  A junction condition by specified homogenization and application to traffic lights.
\newblock {\em Analysis and PDE} 8 (8): 1891--1929, 2015.

\bibitem{Gangbo}
Wilfrid Gangbo, Wuchen Li and Chenchen Mou.
\newblock Geodesics of minimal length in the set of probability measures on graphs.\\
\newblock {\em ESAIM Control Optim. Calc. Var.} 25 (78), 36 pp., 2019.

\bibitem{GU} Olivier Gu\'{e}ant.
\newblock Existence and uniqueness result for mean field games with congestion effect on graphs.
\newblock {\em Applied Mathematics \& Optimization} 72 (2): 291--303, 2015.


\bibitem{LOST}
 V.Laschos, K. Obermayer, Y. Shen, W. Stannat.
\newblock A Fenchel-Moreau-Rockafellar type theorem on the Kantorovich-Wasserstein space with applications in partially observable Markov decision processes.
\newblock {\em J.  Math. Anal. Appl.} 477 (2): 1133--1156, 2019.


\bibitem{Math82}
John~N. Mather.
\newblock {Existence of quasiperiodic orbits for twist homeomorphisms of
              the annulus}.
\newblock {\em Topology}, 21 (4): 457--467, 1982.



\bibitem{Mather91}
John~N. Mather.
\newblock Action minimizing invariant measures for positive definite
  {L}agrangian systems.
\newblock {\em Math. Z.}, 207 (2): 169--207, 1991.

\bibitem{Maz} Jos\`{e} Mazon, Julio D. Rossi and Juli\`{a}n Toledo.
\newblock Optimal Mass Transport on Metric Graphs.
\newblock  {\em SIAM J. Control Optim.} 25 (3): 1609--1632, 2015.

\bibitem{Nash}
 John F. Nash.
\newblock The imbedding problem for Riemannian manifolds.
\newblock {\em Ann. of Math. (2)} 63:  20--63, 1956.

\bibitem{PoSi} Marco Pozza and Antonio Siconolfi
\newblock Discounted Hamilton-Jacobi Equations on Networks and Asymptotic Analysis.
\newblock {\em Indiana Un. Math J.}  to appear.

\bibitem{Rockafellar}
R. Tyrrell Rockafellar.
\newblock Convex analysis.
\newblock {\em Princeton Mathematical Series}, No. 28. Princeton University Press, Princeton, N.J., 1970.


\bibitem{Siburg}
Karl Friedrich Siburg.
\newblock {The principle of least action in geometry and dynamics}.
\newblock {\em Lecture Notes in Mathematics}, 1844. Springer-Verlag, Berlin, xii+128 pp., 2004.


\bibitem{SiconolfiSorrentino}
Antonio Siconolfi and Alfonso Sorrentino.
\newblock Global Results for Eikonal Hamilton-Jacobi Equations on Networks
\newblock {\em Analysis and PDE} 1 (11): 171--211, 2018.


\bibitem{Sorrentinobook}
Alfonso Sorrentino.
\newblock Action-minimizing methods in Hamiltonian dynamics. An introduction to Aubry-Mather theory.
\newblock {\em  Mathematical Notes}, 50, Princeton University Press, Princeton, NJ,  xii+115 pp., 2015.


\bibitem{Sunada}
Toshikazu Sunada.
\newblock Topological crystallography. With a view towards discrete geometric analysis.
\newblock {\em Surveys and Tutorials in the Applied Mathematical Sciences}, 6. Springer, Tokyo, xii+229 pp., 2013.


\bibitem{Villani}
C\'{e}dric Villani.
\newblock Optimal transport. Old and new.
\newblock {\em Springer Science \& Business Media}, 338. Springer, Berlin, xxii+ 976 pp., 2009

\end{thebibliography}
\end{document}